\documentclass[10pt,a4paper]{amsart}
\usepackage{fullpage}
\usepackage{mathrsfs}
\usepackage[left=0.9in, right=0.9in, top=1in, bottom=1.3in, includefoot, headheight=13.6pt]{geometry}
\usepackage[colorlinks=true,hyperindex,citecolor=blue,linkcolor=black]{hyperref}
\input{xy}
\usepackage{cite}
\usepackage{tikz-cd}
\usepackage[capitalize]{cleveref}
\usepackage{eucal,times,amsmath,amsthm,amssymb,mathrsfs,stmaryrd,color,enumerate,accents}
\usepackage{tensor}
\usepackage{xypic}
\usepackage{textcomp}
\usepackage{mathtools,bm,eucal} 
\usepackage[T1]{fontenc} 
\usepackage{enumerate,comment,braket,xspace,csquotes} 

 \numberwithin{equation}{subsection}

\theoremstyle{plain}
\newtheorem{thm-intro}{Theorem}
\newtheorem{theorem}{Theorem}[section]
\newtheorem*{thm*}{Theorem}
\newtheorem{lemma}[theorem]{Lemma}

\newtheorem{prop}[theorem]{Proposition}

\newtheorem{coro}[theorem]{Corollary}

\theoremstyle{definition}
\newtheorem{defi}[theorem]{Definition}
\newtheorem{notation}[theorem]{Notation}
\newtheorem{exem}[theorem]{Example}

\newtheorem{rema}[theorem]{Remark}
\newtheorem{construction}[theorem]{Construction}
\theoremstyle{remark}
\numberwithin{equation}{section}

\newcommand{\GL}{\mathrm{GL}}
\newcommand{\GLn}{\mathrm{GL}_n}
\newcommand{\Gk}{G_K}
\newcommand{\Gal}{\mathrm{Gal}}
\newcommand{\anGLn}{\mathbf{GL}_n^\an}

\newcommand{\fib}{\mathrm{fib}}
\newcommand{\dLocSys}{\mathbf R \mathbf{LocSys}_{\ell, n}}
\newcommand{\abdLocSys}{\mathbf R \mathbf{LocSys}_{\ell, n, \Gamma}}
\newcommand{\LocSys}{\mathbf{LocSys}_{\ell, n}}
\newcommand{\abLocSys}{\mathbf{LocSys}_{\ell, n, \Gamma}}
\newcommand{\LocSysfr}{\LocSys^{\mathrm{framed}}}
\newcommand{\abLocSysfr}{\abLocSys^{\mathrm{framed}}}

\newcommand{\sm}{\mathrm{sm}}
\newcommand{\tr}{\mathrm{tr}}

\newcommand{\Map}{\mathrm{Map}}
\newcommand{\Hom}{\mathrm{Hom}}
\newcommand{\ad}{\mathrm{ad}}

\newcommand{\Ok}{k^\circ}

\DeclareMathOperator{\Spf}{Spf}
\DeclareMathOperator{\Sp}{Sp}
\newcommand{\st}{\mathrm{st}}

\newcommand{\trun}{\mathrm{t}}

\newcommand{\cEnd}{\cE \mathrm{nd}}
\newcommand{\End}{\mathrm{End}}
\newcommand{\Mat}{\mathrm{Mat}}
\newcommand{\unr}{\mathrm{unr}}
\newcommand{\Frac}{\mathrm{Frac}}

\newcommand{\tamepi}{\pi_1^{\mathrm{t}}}
\newcommand{\wildpi}{\pi_1^w}
\newcommand{\tame}{\mathrm{tame}}

\newcommand{\dSt}{\mathrm{dSt}}

\newcommand{\fc}{\mathrm{fc}}
\newcommand{\Sh}{\mathrm{Sh}}
\newcommand{\Ad}{\mathrm{Ad}}

\newcommand{\Def}{\mathbf{\mathrm{Def}}}

\newcommand{\cont}{\mathrm{cont}}
\newcommand{\Q}{\mathbb{Q}}
\newcommand{\Ql}{\bQ_\ell}

\newcommand{\smCAlg}{\mathcal{C}\mathrm{Alg}^\sm}
\newcommand{\CAlg}{\mathcal{C}\mathrm{Alg}}
\newcommand{\adCAlg}{\cC \mathrm{Alg}^{\ad}_{\Ok}}

\newcommand{\AnRing}{\mathrm{AnRing}_k}

\newcommand{\Cat}{\mathcal{C}\mathrm{at}_\infty}
\newcommand{\Coh}{\mathrm{Coh}^+}

\newcommand{\op}{\mathrm{op}}

\newcommand{\Psh}{\mathrm{PShv}}
\newcommand{\Shv}{\mathrm{Shv}}
\newcommand{\ind}{\mathrm{Ind}}
\newcommand{\pro}{\mathrm{Pro}}
\newcommand{\Perf}{\mathrm{Perf}}
\newcommand{\infcat}{$\infty$-category\xspace}
\newcommand{\infcats}{$\infty$-categories\xspace}

\newcommand{\PerfSys}{\mathbf{PerfSys}_{\ell}}
\newcommand{\rigCat}{\mathcal{C}\mathrm{at}^{\mathrm{st}, \omega, \otimes}_{\infty}}

\newcommand{\et}{\text{\'et}}
\newcommand{\emphet}{\emph{\text{\'et}}}
\newcommand{\Sym}{\mathrm{Sym}}
\newcommand{\dR}{\mathrm{dR}}

\newcommand{\rig}{\mathrm{rig}}
\newcommand{\alg}{\mathrm{alg}}
\newcommand{\Fun}{\mathrm{Fun}}

\newcommand{\Mod}{\mathrm{Mod}}
\newcommand{\St}{\mathrm{St}}

\newcommand{\Afd}{\mathrm{Afd}}
\newcommand{\Afdl}{\mathrm{Afd}_{\Q_\ell}}
\newcommand{\An}{\mathrm{An}}
\newcommand{\Anl}{\mathrm{An}_{\Q_\ell}}
\newcommand{\dAnl}{\mathrm{dAn}_{\Q_\ell}}
\newcommand{\dAfd}{\mathrm{dAfd}}
\newcommand{\dAfdl}{\mathrm{dAfd}_{\Q_\ell}}
\newcommand{\dAn}{\mathrm{dAn}}

\newcommand{\dfSch}{\mathrm{dfSch}}

\newcommand{\an}{\mathrm{an}}

\DeclareMathOperator*{\colim}{colim}

\newcommand{\rmB}{\mathrm{B}}
\newcommand{\rmC}{\mathrm C}

\newcommand{\rmH}{\mathrm H}

\newcommand{\rmP}{\mathrm P}

\newcommand{\rmT}{\mathrm T}

\newcommand{\bA}{\mathbb A}

\newcommand{\bE}{\mathbb E}
\newcommand{\bF}{\mathbb F}
\newcommand{\bL}{\mathbb L}
\newcommand{\bQ}{\mathbb Q}

\newcommand{\bT}{\mathbb T}
\newcommand{\bZ}{\mathbb Z}

\newcommand{\fm}{\mathfrak{m}}
\newcommand{\fl}{\mathfrak l}

\newcommand{\cC}{\mathcal C}
\newcommand{\cE}{\mathcal E}
\newcommand{\cF}{\mathcal{F}}

\newcommand{\cM}{\mathcal M}

\newcommand{\cO}{\mathcal{O}}

\newcommand{\cT}{\mathcal{T}}
\newcommand{\cX}{\mathcal X}
\newcommand{\cY}{\mathcal Y}
\newcommand{\cZ}{\mathcal Z}
\newcommand{\cS}{\mathcal S}

\newcommand{\hbZ}{\widehat{\bZ}}
\newcommand{\overK}{\overline{K}}

\usepackage[english]{babel}

\author{Jorge Ant\'onio}
\address{Jorge Ant\'onio,  IMT Toulouse, 118 Rue de Narbonne  31400 Toulouse}
\email{jorge\_tiago.ferrera\_antonio@math.univ-tls.fr}

\begin{document}

\title{Moduli of $\ell$-adic pro-\'etale local systems for smooth non-proper schemes}

\date{\today}

\maketitle

\renewcommand\labelitemi{\textbullet}

\markright{MODULI OF $\ell$-ADIC REPRESENTATIONS}

\begin{abstract}
Let $X$ be a smooth scheme over an algebraically closed field. When $X$ is proper, it was proved
in \cite{me1} that the moduli of $\ell$-adic continuous representations of $\pi_1^\et(X)$, $\LocSys(X)$, is representable by a (derived) $\Ql$-analytic space.
However, in the non-proper case one cannot expect that the results of \cite{me1} hold mutatis mutandis. Instead, assuming $\ell$ is invertible in $X$, one has to bound the ramification at infinity of those considered continuous representations.

The main
goal of the current text is to give a
proof of such representability statements in the open case.
We also extend the representability results of \cite{me1}. More specifically, assuming $X$ is assumed to be proper, we show that $\LocSys(X)$ admits a canonical
shifted symplectic form and we give some applications of such existence result.
\end{abstract}

\setcounter{tocdepth}{1}
\tableofcontents

\section{Introduction}
\subsection{The goal of this paper} Let $X$ be a smooth scheme over an algebraically closed field $k$ of positive characteristic $p>0$.
Without the properness assumption
the \'etale homotopy group $\pi_1^\et(X)$ fits in a short exact sequence of profinite groups
    \begin{equation} \label{ss_intro}
        1 \to \pi_1^w(X) \to \pi_1^\et(X) \to \pi_1^\tame(X) \to 1,
    \end{equation}
where $\pi_1^w(X)$ and $\pi_1^\tame(X)$ denote the \emph{wild} and \emph{tame} fundamental groups of $X$, respectively. One can prove that the profinite group $\pi_1^\tame(X)$ is
topologically of profinite type. However, the profinite group $\pi_1^\et(X)$ is, in general, a profinite pro-$p$ group satisfying no finiteness condition or whatsoever. Needless to say,
the \'etale fundamental group $\pi_1^\et(X)$ will in general not admit a finite number of topological generators.
Consider $X = \bA^1_k$, the affine line. Its
\'etale and wild fundamental groups agree, but they are not topologically of finite type.

For this reason, the main results of \cref{me1} do not apply for a general smooth scheme $X$. In particular, one cannot expect that the moduli of $\ell$-adic continuous representations of
$X$, $\LocSys(X)$, is representable by a $\Q_\ell$-analytic stack. 
The purpose of the current text, is to study certain moduli substacks of $\LocSys$ parametrizing continuous representations
    \[
        \rho \colon \pi_1^\et (X) \to \GLn(A), \quad A \in \Afdl
    \]
such that the restriction $\rho_{\vert \wildpi(X)}$ factors through a finite quotient $p_\Gamma \colon \wildpi(X) \to \Gamma$. Denote $\abLocSys(X)$ such stack. Our main result is the
following:

\begin{theorem} \label{intro:main}
The moduli stack $\abLocSys(X) \colon \Afdl \to \cS$ can be promoted naturally to a derived moduli stack 
    \[
        \abdLocSys(X) \colon \dAfdl \to \cS
    \]
which is representable by a derived
$\Q_\ell$-analytic stack. Given $\rho \in \abdLocSys(X)$, the analytic cotangent complex $\bL_{\abdLocSys, \rho} \in \Mod_{\Ql}$ is naturally equivalent to
    \[
        \bL^\an_{\abdLocSys (X), \rho } \simeq \rmC^*_\emphet \big( X, \Ad ( \rho) \big)^\vee[-1]
    \]
in the derived \infcat $\Mod_{\Ql}.$
\end{theorem}

In particular, \cref{intro:main} implies that the inclusion morphism of stacks
    \[
        j_\Gamma \colon \abdLocSys(X) \hookrightarrow \dLocSys(X)
    \]
induces an equivalence on contangent complexes, in particular it is an \'etale morphism.
We can thus regard $\abdLocSys(X)$ as an admissible substack of $\dLocSys$, in the
sense of $\Q_\ell$-analytic geometry.

The knowledge of the analytic cotangent complex allow us to have a better understanding of the local geometry of $\dLocSys$. In particular, given a continuous representation
    \[
        \overline{\rho} \colon \pi_1^\et(X) \to \GLn(\overline{\bF}_\ell)
    \]
one might ask how $\overline{\rho}$ can be deformed into a continuous representation $\rho \colon \pi_1^\et (X) \to \GLn(\overline{Q}_\ell)$. This amounts to understand the formal moduli
problem $\Def_{\overline{\rho}} \colon \CAlg_{\bF_\ell}^\sm \to \cS$ given on objects by the formula
    \[
        A \in \smCAlg \mapsto \Map_{\cont} \left(\Sh_\et(X), \rmB \GLn(A) \right) \times_{ \Map_{\cont} \left(\Sh_\et(X), \rmB \GLn(\bF_\ell) \right)} \{ \rho \}
        \in \cS,
    \]
where $\Sh_\et(X) \in \pro \big( \cS^{\fc} \big)$ denotes the \'etale homotopy type of $X$. Given $\overline{\rho}$ as above, the functor $\Def_{\overline{\rho}}$ was first
considered by Mazur in \cite{mazurDG}, for Galois representations, in the discrete case. More recently, Galatius and Venkatesh studied its derived structure in detail, see
\cite{galatius_dg} for more details.

One can prove, using similar methods to those in \cite{me1} that the tangent complex
of $\Def_{\overline{\rho}}$ is naturally equivalent to
    \[
        \bT_{\Def_{\overline{\rho}} } \simeq \rmC^*_\et \big(X, \Ad(\rho) \big)[1],
    \]
in the derived \infcat $\Mod_{\overline{\bF}_\ell}$. We can consider
$\Def_{\overline{\rho}}$ as a derived $W(\bF_\ell)$-adic scheme which is locally
admissible, in the sense of \cite{me2}. Therefore, one can consider its rigidification
    \[
        \Def^\rig_{\overline{\rho}} \in \dAnl.
    \]
By construction, we have a canonical inclusion functor
    \[
         j_{\overline{\rho}} \colon \Def^\rig_{\overline{\rho}} \to \LocSys(X).
    \]  
By comparing both analytic cotangent complexes, one arrives at the following result:

\begin{prop} \label{intro:disj_u}
The morphism of derived stacks
    \[
       j_{\overline{\rho}} \colon \Def^\rig_{\overline{\rho}} \to \LocSys(X)
    \]
exhibits $\Def^\rig_{\overline{\rho}}$ as an admissible open substack of
$\LocSys(X)$.
\end{prop}

\cref{intro:disj_u} implies, in particular, that $\LocSys(X)$ admits as an admissible analytic substack the disjoing union $\coprod_{\overline{\rho}} \Def^\rig_{\overline{\rho}}$, indexed
by the set of contininuous representations $\overline{\rho} \colon \pi_1^\et(X) \to \GLn(\overline{\Q}_\ell).$ Nonetheless, the moduli $\LocSys(X)$ admits more (analytic) points in general
than those contained in the disjoint union $\coprod_{\overline{\rho}} \Def^\rig_{\overline{\rho}}$. This situation renders difficult the study of trace formulas on $\LocSys(X)$ which was the
first motivation for the study of such moduli. Ideally, one would like to ''glue'' the connected components of $\LocSys(X)$ in order to have a better behaved global geometry. More 
specifically,
one would like to exhibit a moduli algebras or analytic stack $\cM_{\ell, n}(X)$ of finite type over $\Q_\ell$ such that the space closed points $\cM_{\ell, n}(X)(\overline{\Q}_\ell) \in 
\cS$ would correspond to
continuous $\ell$-adic representations of $\pi_1^\et(X)$. Moreover, one should expect such moduli stack to have a natural derived structure which would provided an understanding of
deformations
of $\ell$-continuous representations $\rho$.

Such principle has been largely successful for instance in the context of continuous $p$-adic representations of a Galois group of a local field of mixed characteristic $(0,p)$. Via $p$-adic
Hodge structure and a scheme-image construction provided in \cite{gee_sch_image}, the authors consider the moduli of Kisin modules which they prove to be an ind-algebraic stack admitting 
strata
given by algebraic stacks of Kisin modules of a fixed height. Unfortunately, the methods used in \cite{gee_sch_image}, namely the scheme-image construction, do not directly generalize to
the
derived setting. Recent unpublished work of M. Porta and V. Melani regarding formal loop stacks might provide an effective answer to this problem, which we pretend to explore in the near
future. However, to the best of the author's knowledge,
there is no other successful attempts outside the scope of $p$-adic Hodge theory.

We will also study the existence of a $2-2d$-shifted symplectic form on $\LocSys(X)$, where $d = \dim  X$. Even though $\LocSys(X)$ is not an instance of an analytic mapping stack it behaves
as such. We need to introduce the moduli stack $\PerfSys(X)$ which corresponds to the moduli of objects associated to the $\rigCat$-valued moduli stack given on objects
by the formula
    \[
        Z \in \dAfdl \mapsto \Fun_{\cE \Cat} \left( \vert X \vert_\et, \Perf \big( \Gamma(Z ) \big) \right)
    \]
where $\cE\Cat$ denotes the \infcat of (small) $\ind( \pro(\cS))$-enriched \infcats. We are then able to prove:

\begin{theorem} \label{intro:shifted}
The derived moduli stack $\PerfSys(X)$ admits a natural shifted symplectic form $\omega$. Explicitly, given $\rho \in \PerfSys(X)$
$\omega$ induces a non-degenerated pairing
    \[
        C^*_\emphet \big(X, \Ad(\rho) \big)[1] \otimes C^*_\emphet \big(X, \Ad(\rho ) \big)[1]
        \to \Ql [2-2d],
    \]
which agrees with Poincar\'e duality.
\end{theorem}

By transport of strucure, the substack $\LocSys(X) \hookrightarrow \PerfSys(X)$ can be equipped with a natural shifted sympletic structure. By restricting further, we equip
the $\abLocSys(X)$ with a shifted symplectic form $\omega_\Gamma$.

\subsection{Summary}
Let us give a brief review of the contents of each section of the text.
Both \S 2.1 and \S 2.2 are devoted to review the main aspects of ramification theory for local fields and smooth varities in positive characteristic. Our exposition is classical and we do 
not pretend
to prove anything new in this context. In \S 2.3
we construct the (ordinary) \emph{moduli stack of continuous $\ell$-adic representations}. Our construction follows directly the methods applied in \cite{me1}. Given $q
\colon \wildpi(X) 
\to \Gamma$ a continuous group homomorphism whose target is finite we construct the moduli stack $\abLocSys(X)$ parametrizing $\ell$-adic continuous representations of $\pi_1^\et(X)$
such that $\rho_{\vert \wildpi(X)} $ factors through $\Gamma$.
We then show that $\abLocSys$ is representable
by a $\Q_\ell$-analytic stack (the analogue of an Artin stack in the context of $\Q_\ell$ analytic geometry).

In \S 3, we show that both the $\Q_\ell$-analytic stacks $\LocSys(X)$ and $\abLocSys(X)$ can be given natural derived structures and we compute their corresponding cotangent complexes. 
It follows then by \cite[Theorem 7.1]{porta_rep} that $\abLocSys(X)$ is representable by a derived $\Q_\ell$-analytic stack.

\S 4 is devoted to state and prove certain comparison results. We prove \cref{intro:disj_u} and relate this result to the moduli of pseudo-representations introduced in \cite{chenevier}.

Lastly, in \S 5 we study the existence of a shifted symplectic form on $\LocSys(X)$. We state and prove \cref{intro:shifted} and analysize some of its applications.

\subsection{Convention and Notations} Throught the text we will employ the following 
notations:

\begin{enumerate}
    \item $\Afdl$ and $\dAfdl$ denote the \infcats of ordinary $\Q_\ell$-affinoid spaces 
    and derived $\Ql$-affinoid spaces, respectively.;
    \item $\Anl$ and $\dAnl$ denote the \infcats of analytic $\Ql$-spaces and derived
    $\Ql$-analytic spaces, respectively;
    \item We shall denote $\cS$ the \infcat of spaces and $\ind(\pro(\cS)) \coloneqq \ind \big( \pro \big(\cS \big)
    \big)$ the \infcat of ind-pro-objects on $\cS$.
    \item $\Cat$ denotes the \infcat of small \infcats and $\cE \Cat$ the \infcat
    of $\ind(\pro(\cS))$-enriched \infcats.
    \item Given a continuous representation $\rho$, we shall denote $\Ad ( \rho) \coloneqq
    \rho \otimes \rho^\vee$ the corresponding adjoint representation;
    \item Given $Z \in \Afdl$ we sometimes denote $\Gamma(Z) \coloneqq \Gamma(Z)$ the derived
    $\Ql$-algebra of global sections of $Z$.
\end{enumerate} 

\subsection{Acknowledgements} I am grateful to Jean-Baptiste Tessier and Bertrand To\"en for many useful discussions and suggestions on the contents of the present text. I would also like to acknowledge le \emph{S\'eminaire Groupes R\'eductifs et formes
automorphes} for the invitation to expose many of my ideas about the subject.

\section{Setting the stage}

\subsection{Recall on the monodromy of (local) inertia} In this subsection we recall some well known facts on the monodromy of the local inertia, our exposition follows closely \cite[\S 1.3]{fontaine_ouyang}.

Let $K$ be a local field, $\cO_K$ its ring of integers and $k$ the residue field which we assume to be of characteristic $p>0$ different from $\ell$. Fix $\overline{K}$ an algebraic closure of $K$ and denote by $\Gk \coloneqq
 \Gal \left( \overK / K \right)$ its absolute Galois group.
 
 \begin{defi}
Given a finite Galois extension $L/K$ with Galois group $\Gal \left( L / K \right)$ we define its \emph{inertia group}, denoted $I_{L/K} $, as the subgroup of $\Gal \left( L / K \right)$
spanned by those elements of $\Gal \left( L /K \right)$ which act trivially on $\mathfrak l \coloneqq \cO_L / \fm_L$, where $\O_L $ denotes the ring of integers of $L$ and $ \mathfrak{m}_L$ the corresponding maximal ideal. 
\end{defi}

\begin{rema}
We can identify the inertia subgroup $ I_{L/K}$ of $\Gal( L / K )$ with the kernel of the surjective continuous group homomorphism $q \colon \Gal( L / K ) \to \Gal( l/ k)$. We have thus a short exact sequence of profinite groups
	\begin{equation} \label{inertia}
		1 \to I_{L/K} \to \Gal( L / K ) \to \Gal( l/ k) \to 1.
	\end{equation}
In particular,
we deduce that the inertia subgroup $I_{L / K }$ can be identified with a normal subgroup of $\Gal( L / K)$.
\end{rema}

\begin{rema}
Letting the field extension $L/K$ vary, we can assemble
together the short exact sequences displayed in \eqref{inertia} thus obtaining a short exact sequence of profinite groups
	\begin{equation} \label{absinert}
		1 \to I_K \to \Gk \to G_k \to 1,
	\end{equation}
where $G_k \coloneqq \Gal( \overline{k}/ k)$ where $\overline{k}$ denotes the algebraic closure of $k$ determined by $ \overline{K}$.
\end{rema}

\begin{defi}[Absolute inertia]
Define the \emph{(absolute) inertia group of $K$} as the inverse limit
	\[
		I_K := \lim_{L/K \text{ finite}} I_{L/ K},
	\]
which we canonically identify with a subgroup of $\Gk$.
\end{defi}

\begin{defi}[Wild inertia]
Let $L/ K$ be a field extension as above. We let $P_{L/K}$ denote the subgroup of $I_{L/K}$ which
consists of those elements of $I_{L/K}$ acting trivially on $\cO_L/ \fm^2_L$. We refer to $P_{L/K}$ as the \emph{wild inertia group} associated to $L/K$.
\end{defi}

\begin{defi}[Absolute wild inertia]
We define the absolute wild inertia group of $K$ as:
	\[
		P_K : = \lim_{L \text{ finite}} P_{L/K}.
	\] 
\end{defi}

\begin{rema}
We can identify the absolute wild inertia group $P_K $ with a normal subgroup of $I_K$.
\end{rema}

Consider the exact sequence
	\begin{equation} \label{wild_tame}
		1 \to P_K \to I_K \to I_K / P_K \to 1.
	\end{equation}
Thanks to \cite[Lemma 53.13.6]{stacks} it follows that the wild inertia group $P_K$ is a \emph{pro-$p$} group. When $K = \Q_p$ a theorem of Iwasawa implies that $P_K$ is not topologically of finite generation, even though
$G_K$ is so.
Nonetheless, the quotient $I_K / P_K$ is much more amenable:

\begin{prop}{\cite[Corollary 13]{bommel}} \label{tame_mon} Let $p \coloneqq \mathrm{char}(k)$ denote the residual characteristic of $K$.
The quotient $I_K / P_K$ is canonically isomorphic to $\hbZ'(1)$, where the latter denotes the profinite group $\prod_{q \neq p} \bZ_q(1)$. In particular, the quotient profinite group $I_K / P_K$ is topologically of finite generation.
\end{prop}

Define $P_{K, \ell} $ to be the inverse image of $\prod_{q \neq \ell, p } \bZ_q$ in $I_K$. We have then a short exact sequence of profinite groups
	\[
		1 \to P_K \to P_{K, \ell} \to \prod_{q \neq \ell, p } \bZ_q \to 1.
	\]
Define similarly $G_{K, \ell} \coloneqq G_K / P_{K, \ell}$ the quotient of $G_K$ by $P_{K, \ell}$. We have a short exact sequence of profinite groups
	\begin{equation} \label{e1}
		1 \to P_{K, \ell} \to \Gk \to G_{K, \ell} \to 1.
	\end{equation}
Assembling together \eqref{wild_tame} and \cref{tame_mon} we obtain a short exact sequence
	\begin{equation} \label{e2}
		1 \to \bZ_\ell(1) \to G_{K, \ell} \to G_k \to 1.
	\end{equation}
	
\begin{rema}
As a consequence of both \eqref{e1} and \eqref{e2} the quotient $G_{K, \ell}$ is topologically of finite type.
\end{rema}

Suppose we are now given a continuous representation
	\[
		\rho \colon \Gk \to \GLn (E_\ell),
	\]
where $E_\ell$ denotes a finite field extension of $\Ql$. Up to conjugation, we might assume that $\rho$ preserves a lattice of $E_\ell$. More explicitly, up to conjugation we have a commutative diagram of the form
	\[
	\begin{tikzcd}
		G_K \ar{r}{\widetilde{\rho}} \ar{dr}[swap]{\rho} & \GLn(\bZ_\ell) \ar{d} \\
						& \GLn(\Q_\ell)
	\end{tikzcd}.
	\]
	
Therefore $\widetilde{\rho} \left( G_K \right) $ is a closed subgroup of $\GLn(\bZ_\ell)$. Consider the short exact sequence
	\[
		1 \to N_1 \to \GLn(\bZ_\ell) \to \GLn(\bF_\ell) \to 1,
	\]
where $N_1$ denotes the group of $\GLn(\bZ_p)$ formed by congruent to $\mathrm{Id}$ mod $\ell$ matrices. In particular, $N_1$ is a profinite pro-$\ell$ group.
By construction, every finite quotient of $P_{K, \ell}$ is of order prime to $\ell$. One then has necessarily
	\[
		\rho \left( P_{K, \ell}  \right) \cap N_1 = \{1 \}.
	\] 
As a consequence, the group $\rho( P_{K, \ell})$ injects into the finite group $\GLn(\bF_\ell)$ under $\rho$.
Which in turn implies that the (absolute) wild inertia group $P_K$ itself acts on $\GLn(\Q_\ell)$ via a finite quotient. 

\subsection{Geometric \'etale fundamental groups}
Let $X$ be a geometrically connected smooth scheme over an algebraically closed field $k$ of positive
characteristic.
Fix once and for all a geometric point $\iota_x \colon \overline{x} \to X$ and consider the corresponding \'etale fundamental group $\pi_1^{\et}(X) \coloneqq \pi_1^{\et}(X, \overline{x})$, a profinite group. If we assume that $X$ is moreover proper one has the following
classical result:

\begin{theorem}{\cite[Expos\'e 10, Thm 2.9]{grothendieckSGA1}} \label{proper_case}
Let $X$ be a smooth and proper scheme over an algebraically closed field.
Then its \'etale fundamental group $\pi^\emph{\et}_1 \left( \overline{X} \right)$ is topologically of finite type.
\end{theorem}

Unfortunately, the statement of \cref{proper_case} does not hold in the non-proper case as the following proposition illustrates:

\begin{prop}
Let $k$ be an algebraically closed field of positive characteristic. Then the \'etale fundamental group of the affine line $\pi_1^\emph{\et}(\mathbb A^1_k)$ is not topologically finitely generated.
\end{prop}

\begin{proof}
For each integer $n \geq 1$, one can exhibit Galois covers of $\mathbb A^1_k$ whose corresponding automorphism group is isomorphic to $\left( \mathbb Z / p \mathbb Z \right)^n$. This statement readily
implies that $\pi_1^{\et}(\mathbb A^1_k)$ does not admit a finite
number of
topological generators. In order to construct such coverings, we consider the following endomorphism of the affine line
	\[
		\phi_n \colon \mathbb A^1_k \to \mathbb A^1_k,
	\]
defined via the formula
	\[
		\phi_n \colon x \mapsto x^{p^n} - x.
	\] 
The endormophism $\phi_n$ respects the additive group structure on $\mathbb A^1_k$. Moreover,
the differential of $\phi_n$ equals $-1$. For this reason, $\phi_n$ induces an isomorphism on cotangent spaces and, in particular, it is an \'etale morphism.
As $k$ is algebraically closed, $\phi_n$ is surjective and it is finite, thus a finite \'etale covering. The automorphism group
of
$\phi_n$ is naturally
identified with its kernel, which is isomorphic to $\mathbb F_{p^n}$. The statement of the proposition now follows.
\end{proof}

\begin{defi}
Let $G$ be a profinite group and $p$ a prime numbger, we say that $G$ is \emph{quasi-$p$} if $G$ equals the subgroup generated by all $p$-Sylow subgroups of $G$.
\end{defi}

Examples of quasi-$2$ finite groups include the symmetric groups $S_n$, for $n \geq 2$. Moreover, for each prime $p$, the group $\mathrm{SL}_n(\mathbb F_p)$ is quasi-$p$.
Let $X = \mathbb A^1_k$ be the affine line over an algebraically closed field $k$ of characteristic $p>0$. 
We have the following result proved by Raynaud which was originally a conjecture of Abhyankar:

\begin{theorem}{\cite[Conjecture 10]{Clark}} 
Every finite quasi-$p$ group can be realized as a quotient of $\pi_1^{\emph{\et}} \left(X \right)$.
\end{theorem}

\begin{rema}
In the example of the affine line the infinite nature of $\pi_1(\mathbb A^1_k )$ arises as a phenomenon of the existence of \'etale coverings whose ramification at infinity can be as large as we desire. This phenomenon is special to the positive characteristic
setting.
Neverthless, we can prove that $\pi_1^{\et}(X)$ admits a topologically finitely generated quotient which corresponds to the group of automorphisms of tamely ramified coverings. On the other hand, in the proper case every finite \'etale
covering of $X$ is everywhere unramified.
\end{rema}

\begin{defi}
Let $X \hookrightarrow \overline{X}$ be a normal compactification of $X$, whose existence is guaranteed by \cite{nagata}. Let
$f \colon Y \to X$ be a finite \'etale cover with connected source.
We say that $f$ is \emph{tamely ramified along} the divisor $D : = \overline{X} \backslash X$ if every codimension-$1$ point $x \in D$ is
tamely ramified in the corresponding extension field extension $k(Y) / k(X)$. 
\end{defi}

\begin{prop}
Tamely ramified extensions along $D \coloneqq \overline{X} \backslash X$ of $X$ are classified by a quotient $\pi_1^{\emph{\et}}(X) \to \tamepi(X, D)$,
referred to as the \emph{tame fundamental group of $X$ along $D$}.
\end{prop}

\begin{rema} Let $\overline{X}$ denote a smooth compactification of $X$ and $D \coloneqq \overline{X} \backslash X.$
We denote by $\wildpi(X, D)$, the \emph{wild fundamental group of $X$ along $D$},
the kernel of the continuous morphism $ \pi_1^{\et}(X) \to \tamepi(X)$.
\end{rema}

\begin{defi} \label{}
Assume $X$ is a normal connected scheme over $k$.
\begin{enumerate}
\item Let $f \colon Y \to X$ be an \'etale covering. We say that $f$ is \emph{divisor-tame} if for every normal compactification
$X \hookrightarrow \overline{X}$, $f$ is tamely ramified along $D = \overline{X} \backslash X$.
\item Let $f \colon Y \to X$ be an \'etale covering. We shall refer to $f$ as \emph{curve-tame} if for every smooth curve $C$ over $k$ and morphism
$g \colon C \to X$, the base change $Y \times_{X} C \to C$ is a tame covering of the curve $C$.
\end{enumerate}
\end{defi}

\begin{rema}
In \cref{} $X$ is assumed to be a normal connected scheme over a field of positive characteristic. Currently, we lack a
resolution of singularities theorem in this
setting. Therefore, a priori, one cannot expect that both divisor-tame and curve-tame notions agree in general. Indeed, one can expect
many regular normal crossing compactifications of $X$ to exist, or none. 
\end{rema}

Neverthless, one has the following result:

\begin{prop}{\cite[Theorem 1.1]{kerz_tame}}
Let $X$ be a smooth scheme over $k$ and let $f \colon Y \to X$ be an \'etale covering. Then $f$ is divisor-tame if and only if it is curve-tame.
\end{prop}

\begin{defi}
The \emph{tame fundamental group} $\tamepi(X)$ is defined as the quotient of $\pi_1^\et(X)$ by the normal closure of
opens subgroup of $\pi_1^\et(X)$ generated by the wild fundamental groups  $\pi_1^w(X, D)$ along $D$,
for each normal compactification $X \hookrightarrow \overline{X}$. 
\end{defi}

\begin{rema} The notion of tameness is stable under arbitrary base changes between smooth schemes. In particular, given
a morphism $f \colon Y \to X$ between smooth schemes over $k$, one has a functorial well defined morphism $\tamepi (Y) \to \tamepi(X)$ fitting in a commutative
diagram of profinite groups
    \[
    \begin{tikzcd}
        \pi_1^\et(Y) \ar{r} \ar{d} & \pi_1^\et(X) \ar{d} \\
        \tamepi(Y) \ar{r} & \tamepi(X).
    \end{tikzcd}
    \]
Moreover, the profinite group $\tamepi(X)$ classifies tamely ramified \'etale coverings of $X$.
\end{rema}

\begin{rema}
The tame fundamental group $\tamepi(X)$ classifies finite \'etale coverings $f \colon X \to Y$ which are tamely ramified
along any divisor at infinity.
\end{rema}

\begin{defi}
We define the \emph{wild fundamental group} of $X$, denoted $\pi_1^w(X)$, as the kernel of the surjection $\pi_1^{\et}(X) \to
\tamepi(X)$. It is an open normal subgroup of $\pi_1^{\et}(X)$.
\end{defi}

\begin{prop}{\cite{Clark}}
Let $C$ be a geometrically connected smooth curve over $k$. Then the wild fundamental group $\pi_1^w(C)$ is a pro-$p$-group.
\end{prop}

\begin{theorem}{\cite[Appendix 1, Theorem 1]{cadoret}} \label{cadoret}
Let $X$ be a smooth and geometrically connected scheme over $k$. There exists a smooth, geometrically connected curve
$C/ k$ together with a morphism $f \colon C \to X$ of varieties such that the corresponding morphism at the level of
fundamental groups $\pi_1^{\emph{\et}}(C) \to \pi_1^{\emph{\et}}(X) \to \tamepi(X)$ is surjective and it factors by a well defined morphism
$\pi_1^{\mathrm{t}}(C) \to \tamepi(X)$. In particular, $\tamepi(X)$ is topologically finitely generated.
\end{theorem}

\begin{rema}
\cref{cadoret} implies that $\tamepi (\mathbb A^1_k)$ admits a finite number
of topological generators. In fact, the group $\tamepi(\mathbb A^1_k)$ is trivial.
\end{rema}

\subsection{Moduli of continuous $\ell$-adic representations} \label{section 2.3}
In this \S, $X$ denotes a smooth scheme over an algebraically closed field of positive characteristic $p > 0$. Nevertheless, our arguments apply when $X$ is
the spectrum of a local field of mixed characteristic.

\begin{rema}
Let $A \in \Afd$ be $\Q_{\ell}$-affinoid algebra $A \in \Afd$. It admits a natural topology induced from a choice of a norm on $A$, compatible with the usual $\ell$-adic valuation on $\Q_\ell$. 
Given $\mathbf G$
an analytic $\Q_\ell$-group space we can consider the corresponding group of $A$-points on $\mathbf G$, $\mathbf G(A)$. The group $\mathbf G (A)$ admits a natural topology induced from the non-archimedean topology on $A$.
In the current text we will be interested in studying the moduli functor parametrizing continuous representations
	\[
		\rho \colon \pi_1^\et(X) \to \anGLn(A).
	\]
Nevertheless, our arguments can be directly applied when we instead consider the moduli of continuous representations 
    \[
        \pi_1^\et(X) \to \mathbf G^\an (A),
    \]
where $\mathbf G$ denotes a reductive group scheme.
\end{rema}

\begin{defi} Let $G$ be a profinite group.
Denote by
	\[
		\LocSysfr(G) \colon \Afdl \to  \mathrm{Set},
	\]
the \emph{functor of rank $n$ continuous $\ell$-adic group homomorphisms of $G$}. It is given on objects by the formula
	\begin{equation} \label{e21}
		A \in \Afd^{\op} \mapsto \Hom_{\mathrm{cont}} \left( G, \GLn( A) \right) \in \mathrm{Set},
	\end{equation}
where the right hand side of \eqref{e21} denotes the set of continuous group homomorphisms $G_K \to \GLn(A)$.
\end{defi}

\begin{notation}
Whenever $G = \pi_1^\et(X)$ we denote $\LocSys(X) \coloneqq \LocSys(\pi_1^\et(X))$.
\end{notation}

\begin{prop}{\cite[Corollary 2.2.16]{me1}} \label{me1}
Suppose $G$ is a topologically finitely generated profinite group. Then the functor $\LocSysfr(G)$
is representable by a $\Q_{\ell}$-analytic space. 
\end{prop}

By the results of the previous \S, the \'etale fundamental group $\pi_1^\et(X)$ is almost never topologically finitely generated
in the non-proper case. For this reason, we cannot expect the functor $\LocSysfr(G_X) $ to be representable by an object in the category $\An_{\Q_\ell}$ of $\Q_\ell$-analytic spaces.
Nevertheless,
we can prove an analogue of \cref{me1} if we consider instead certain subfunctors of $\LocSysfr$. More specifically, given a finite
quotient $q \colon \wildpi(X) \to \Gamma$ we can consider the moduli parametrizing continuous $\ell$-adic representations of $\pi_1^\et(X)$
whose restriction to $\wildpi(X)$ factors through $\Gamma$:

\begin{construction} \label{const1}
Let $q \colon \mathrm \wildpi(X) \to \Gamma$ denote a surjective continuous group homomorphism, whose target is a finite group (equipped with the discrete topology). We define the functor of \emph{continuous group homomorphisms $\pi_1^\et(X)$ to $\GLn(-)$
with $\Gamma$-bounded ramification at infinity}, as the fiber product
	\begin{equation} \label{eq_def}
		\abLocSys^{\mathrm{framed}}(\pi_1^\et(X)): =  \LocSys^{\mathrm{framed}}(\pi_1^\et(X))\times_{  	
		\LocSys^{\mathrm{framed}}(\mathrm \wildpi(X))	}	\LocSys^{\mathrm{frame}}(\Gamma),
	\end{equation}
computed in the category $\Fun \left( \Afd^{\op}, \mathrm{Set} \right)$.
\end{construction}

\begin{rema}
The moduli functor $\abLocSysfr(X)$ introduced in \cref{const1} depends on the choice of the continuous surjective homomorphism $q \colon
\mathrm P_X \to \Gamma$. However, for notational convenience we drop the subscript $q$.
\end{rema}

We have the following result:

\begin{theorem} \label{hom_loc}
The functor $\abLocSysfr(X)$ is representable by a $\Q_{\ell}$-analytic stack.
\end{theorem}

\begin{proof} Let $r$ be a positive integer and denote $\mathrm F^{[r]}$ a free profinite group on $r$ topological generators.
The finite group $\Gamma$ and the quotient $G_X / \mathrm P_X $ are topologically of finite generation. Therefore, it is possible to choose
a continuous group homomorphism 
	\[
		p \colon \mathrm F^{[r]} \to \pi_1^\et(X),
	\]
such that the images $p(e_i)$, for $i = 1, \dots, r$, form a set of generators for $\Gamma$, seen as a quotient of $\wildpi(X)$, and for
$\tamepi(X) \cong \pi_1^\et(X) / \wildpi(X)$.
Restriction under $\varphi$ induces a closed immersion of functors 
	\[
		\abLocSysfr(G_X) \hookrightarrow \LocSysfr(\mathrm F^{[r]}).
	\]
Thanks to \cite[Theorem
2.2.15.]{me1}, the latter
is representable by a rigid $\Q_\ell$-analytic space, denoted $X^{[r]}$. It follows that $\abLocSysfr(G_X)$ is representable by a closed subspace of $X^{[r]}$, which proves the statement.
\end{proof}

\begin{defi} \label{const_1}
Let $\Psh \left( \Afdl
\right) \coloneqq \Fun \big( \Afdl^\op, \cS \big)$ denote the \infcat of $\cS$-valued preasheaves on $\Afdl$.
Consider the \'etale site $( \Afd, \tau_{\et})$. We define the \infcat of \emph{higher stacks} on $(\Afd, \tau_{\et})$, $\St \left( \Afd, \tau_{\et} \right),
$ as the full subcategory of $\Psh \left( \Afd \right)$ spanned by those pre-sheaves which satisfying \'etale hyper-descent, \cite[\S 7]{lurieHTT}.
\end{defi}

\begin{rema}
The inclusion functor $ \St \left( \Afdl, \tau_{\et} \right) \subseteq \Psh \left( \Afd \right)$ admits a left adjoint, which is a left localization functor.
For this reason, the \infcat $\St \left( \Afdl, \tau_{\et} \right)$ is a presentable \infcat. One can actually prove that $\St \big( \Afdl, \tau_\et \big)$
is the hypercompletion of the $\infty$-topos of \'etale sheaves on $\Afdl$, $\Shv_\et \big( \Afd \big)$.
\end{rema}

\begin{defi}
Consider the geometric context $(\dAfd, \tau_{\et}, \mathrm P_{\sm} )$, \cite[Definition 2.3.1]{me1}. Let $\St \left( \Afdl, \tau_{\et}, \mathrm \rmP_\sm \right)$ denote the full subcategory of $\St( \Afd, \tau_{\et})$
spanned by geometric stacks, \cite[Definition 2.3.2]{me1}. We will refer to an object $\cF \in  \St \left( \Afdl, \tau_{\et}, \mathrm \rmP_\sm \right)$ as the
a $\Q_\ell$-analytic stack and we refer to $\St \big( \Afdl, \tau_\et \big)$ as the \infcat of \emph{$\Q_\ell$-analytic stacks}.
\end{defi}

\begin{exem}
Let $\mathbf G$ be a group object in the $\infty$-category $\St \left( \Afd, \tau_{\et}, \mathrm P_{\sm} \right)$. Given a $\mathbf G$-equivariant object $\cF \in \St \left( \Afd, \tau_{\et}, P_{\sm} \right)^{\mathbf G}$ we denote $[\cF / \mathbf G ]$ the geometric
realization of the simplicial object
	\[
	\begin{tikzcd}
 		\cdots \arrow[r, shift left=0.5] \arrow[r, shift right=1.5]
		\arrow[r, shift left=1.5]
		\arrow[r, shift right=0.5]
		& \mathbf G^2 \times \cF \arrow[r, shift left] \arrow[r]
		\arrow[r, shift right]
		& \mathbf G \times \cF  \arrow[r, shift left=0.5] \arrow[r, shift right=0.5]
		& \cF	\end{tikzcd}
	\]
computed in the $\infty$-category $\St \big( \Afdl, \tau_{\et} \big)$. We refer to $[\cF / \mathbf G ]$ as the \emph{quotient stack objec}t of $\cF$ by $\mathbf G$. 
\end{exem}

\begin{lemma}{\cite[\S 2.3]{me1}.}
Suppose $ \mathbf G \in \St \left( \Afd, \tau_{\et}, P_{\sm} \right) $ is a smooth group object and $\cF $ is representable by a $\Q_{\ell}$-analytic space. Then the quotient stack object $[\cF/ \mathbf G]$ is representable by a geometric stack.
\end{lemma}

\begin{rema}
The smooth group $\mathbf{GL}^\an_n \in \Anl$ acts by conjugation on the moduli functor $\LocSysfr$. 
\end{rema}

\begin{defi}
Let $\LocSys(X) \coloneqq [\LocSysfr(X) / \mathbf \GLn^{\an}]$ denote the \emph{moduli stack of rank $n$ $\ell$-adic pro-\'etale local systems on $X$}. Given a continuous surjective group homomorphism $q \colon \mathrm \wildpi(X) \to \Gamma$ whose
target is a finite group we define the substack of $\LocSys(X)$ spanned by rank $n$ $\ell$-adic pro-\'etale local systems on $X$ \emph{ramified at infinity by level $\Gamma$} as the fiber product
	\[
		\abLocSys \coloneqq \LocSys(X) \times_{\LocSys ( \wildpi(X)} \LocSys(\Gamma)
	\]
\end{defi}

\begin{theorem} \label{main1}
The moduli stack $\abLocSys(X)$ is representable by a $\Q_{\ell}$-analytic stack.
\end{theorem}

\begin{proof}
We have a canonical map $ \abLocSys^{\mathrm{framed}}(G_X)\to \abLocSys(X)$, which exhibits the former as a smooth atlas of the latter. The result now follows formally, as explained in \cite[\S 2.3]{me1}.
\end{proof}

One can prove that there is an equivalence between the space of continuous representations
	\[
		\rho \colon \pi_1^\et(X) \to \anGLn(A), \quad A \in \Afdl
	\]
and the space of rank $n$ pro-\'etale $A$-local systems on $X$. We thus have the following statement:

\begin{prop}{\cite[Corollary 3.2.5]{me1}}
The functor $\LocSys(X)$ parametrizes pro-\'etale local systems of rank $n$ on $X$.
\end{prop}

\begin{proof}
The same proof of \cite[Corollary 3.2.5]{me1} applies.
\end{proof}

\section{Derived structure} Let $X$ be a smooth scheme over an algebraically closed field $k$ and fix a finite quotient $q \colon \wildpi(X) \to \Gamma$.
In this \S we will study at full the deformation theory of both the $\Q_\ell$-analytic moduli stacks $\LocSys(X)$ and $\abLocSys(X)$. Our goal is to show that $\LocSys(X)$ and
$\abLocSys(X)$ can be naturally promoted to \emph{derived $ \Q_{\ell}$-stacks}, denoted $\dLocSys(X)$ and $\abdLocSys(X)$, respectively. 
Therefore the corresponding $0$-truncations $\trun_{\leq 0} \dLocSys(X)$
and $\trun_{\leq 0} \abdLocSys(X)$ are equivalent to $\LocSys(X)$ and $\abLocSys(X)$, respectively. We will prove moreover that both $\abdLocSys(X)$ and $\LocSys(X)$ admit tangent complexes and give a precise formula for these. Moreover, we show that
the substack $\abdLocSys(X)$ is geometric with respect to the geometric context $\big( \dAfdl, \tau_\et, \rmP_\sm \big)$. In particular, $\abdLocSys(X)$ admits a cotangent complex which we can understand at full.

We compute the corresponding cotangent complexes and analyze some consequences of the existence of derived structures on theses objects.
We will use extensively the language of derived $\Q_{\ell}$-analytic geometry as developed in \cite{porta_der, porta_rep}.

\subsection{Derived enhancement of $\LocSys(X)$} Recall the \infcat of derived $\Ql$-affinoid spaces $\dAfdl$ introduced in \cite{porta_der}.
Given a derived $\Ql$-affinoid space $Z : = (\cZ, \cO_Z) \in \dAfdl$, we denote 
	\[
	\Gamma(Z) \coloneqq \Gamma \left(  \cO_Z^{\alg} \right) \in \CAlg_{\Q_{\ell}}
	\]
the corresponding derived ring of \emph{global sections on $Z$}, see \cite[Theorem 3.1]{porta_hom} for more details.
\cite[Theorem 4.4.10]{me2} implies that $\Gamma(Z)$ always admits a formal model, i.e., a $\ell$-complete derived
$\mathbb Z_{\ell}$-algebra $A_0 \in \CAlg_{\mathbb Z_\ell}$ such that $\left( \Spf A_0 \right)^{\mathrm{rig}} \simeq X$. Here $(-)^{\mathrm{rig}}$ denotes the rigidification functor from derived formal $\mathbb Z_\ell$-schemes to derived
$\Q_\ell$-analytic spaces, introduced in \cite[\S 4]{me2}. This allow us to prove:

\begin{prop}{\cite[Proposition 4.3.6]{me1}} \label{prop:enr}
The \infcat of perfect complexes on $A$, $\mathrm{Perf}(A)$, admits a natural structure of
$\ind \big( \pro \big( \cS \big) \big)$-enriched \infcat, i.e., it can be naturally upgraded to an object in the \infcat $\cE \Cat .$
\end{prop}

\begin{defi}
Let $Y \in  \ind \big( \pro \big( \cS \big) \big)$. We define its \emph{materialization} by the formula
	\[
		\mathrm{Mat} \left(\cX \right) := \Map_{\ind(\pro(\cS))} \left( *, \cX \right) \in \cS ,
	\]
where $* \in \ind(\pro(\cS))$ denotes the terminal object.
This formula is functorial. For this reason, we have a well defined, up to contractible indeterminacy functor, \emph{materialization functor} $\Mat \colon \ind(\pro(\cS)) \to \cS $.
\end{defi}

As a consequence of \cref{prop:enr}, there exists an object $\rmB \cEnd  (Z) \in  \ind \big( \pro \big( \cS \big) \big)$, functorial in $ Z \in \dAfdl$, such that its \emph{materalization} is equivalent to
	\begin{equation} \label{eq:BGLn}
		\mathrm{Mat} \left( \cEnd (Z)  \right)  \simeq \mathrm  \rmB \End(\Gamma(Z)^n) \in  \cS .
	\end{equation}
The right hand side of \eqref{eq:BGLn} denotes the usual Bar-construction applied to $\mathbb E_1$-monoid object $\End(\Gamma(Z))  \in \cS$. Moreover, given $Y \in  \ind \big( \pro (\cS ) \big) $
every continuous morphism
	\[
		 Y \to \rmB \cEnd (Z), \text{in }  \ind \big( \pro \big( \cS \big) \big) 
	\]
is such that its materialization factors as
	\[
		\Mat \left( Y \right) \to \rmB \GLn(\Gamma(Z)) \hookrightarrow \rmB \End(\Gamma(Z))
	\]
in the \infcat  $\cS$.
See \cite[\S 4.3 and \textsection
4.4]{me1} for more details.

\begin{defi}{\cite[Notation 3.6.1]{lurieDAGXIII}}
We shall denote $\Sh^{\et}(X)$ the \emph{\'etale shape of $X$} defined as the fundamental groupoid associated to the $\infty$-topos $\Shv_{\text{\'et}} \left( X \right)^\wedge$, of hyper-complete \'etale sheaves on $X$.
\end{defi}

\begin{defi}
Let $X$ be as above. We define the \emph{derived moduli stack of $\ell$-adic pro-\'etale local systems of rank $n$ on $X$} as the functor
	\[
		\dLocSys(X) \colon \mathrm{dAfd}_{\mathbb{Q}_\ell}^{\op} \to \cS, 
	\]
given informally on objects by the formula 
	\[
		 Z \in  \mathrm{dAfd}_{\mathbb{Q}_\ell}^{\op} \mapsto \lim_{n \geq 0 } \Map_{\mathrm{Mon}_{\mathbb E_1}(\mathcal{C})} \left( \mathrm{Sh}^{\et}(X), \rmB  \cEnd \left( \trun_{\leq n} (Z)  \right)  \right),
	\]
where $\trun_{\leq n} (Z)$ denotes the $n$-th truncation functor on derived $\Q_\ell$-affinoid spaces.
\end{defi}

\begin{notation}
Given $Z \in \dAfdl$ we sometimes prefer to employ the notation 
	\[
		\dLocSys(X)(\Gamma(Z)) \coloneqq \dLocSys(X)(Z).
	\]
Let $\rho \in \dLocSys(X)(\Gamma(Z))$, 
we refer to it as a \emph{continuous representation of }$\Sh^{\et}(X)$ \emph{with coefficients in $\Gamma(Z)$}.
\end{notation}

\begin{defi}
Let $Y  := \lim_m Y_m \in \pro \left( \cS \right)$. Given an integer $n \geq 0$, we define the \emph{$n$-truncation of $Y$} as
	\[
		\tau_{\leq n} \left(  Y \right) \coloneqq \lim_m \tau_{\leq n  } ( Y_m)  \in \pro( \cS_{\leq n} ),
	\]
i.e. we apply pointwise the truncation functor $\tau_{\leq n} \colon \cS \to \cS$ to the diagram defining $Y = \lim_m Y_m \in \pro ( \cS)$.
\end{defi}

$\ind(\pro(\cS))$

\begin{notation}
Let $\iota \colon \Afd \to \dAfdl$ denote the canonical inclusion functor. Denote by
	\[
		\mathrm t_{\leq 0} \left( \dLocSys(X) \right) := \dLocSys(X) \circ \iota,
	\]
the restriction of $\dLocSys(X)$ to $\Afdl$. 

Given $Z \in \Afd^{\op}$, the object $\rmB \cEnd(Z) \in  \ind \big( \pro \big( \cS \big) \big)  $ is $1$-truncated. As a consequence, we have an equivalence of mapping spaces:
	\[
		\Map_{\ind(\pro(\cS))} \big( \Sh^{\et} (X), \rmB \cEnd(Z) \big) \simeq \Map_{\ind(\pro(\cS))} \left( \tau_{\leq 1} \Sh^{\et} (X), \rmB \cEnd(Z)\right).
	\]
We have moreover an equivalence of profinite spaces $\tau_{\leq 1} \Sh^\et(X) \simeq \mathrm B \pi_1^\et(X)$. Given a continuous group homomorphism $\rho \colon \pi_1^\et(X) \to \GLn(A)$ we can associate, via the cobar construction performed in
the \infcat $\cT  \op_{\mathrm{na}}$, a well defined morphism
	\[
		\rmB \rho \colon \rmB \pi_1^\et(X) \to \rmB \cEnd(A),
	\]
in the \infcat $\ind(\pro(\cS))$. This construction provide us with a well defined, up to contractible indeterminacy,
	\[
		p_A \colon \LocSysfr (X)(A) \to \Map_{\ind(\pro(\cS)) } \left(\rmB \pi_1^\et(X), \rmB \cEnd(Z) \right). 
	\]
On the other hand, the morphisms $p_A$ assemble to provide a morphism of stacks
	\[
		p \colon \LocSysfr(X) \to \trun_{\leq 0} \dLocSys(X) .
	\]
\end{notation}

\begin{prop} \label{123}
The canonical morphism
	\[
		p \colon  \LocSysfr(X) \to  \mathrm t_{\leq 0} \dLocSys(X) ,
	\]
in the \infcat $\St \big( \Afdl, \tau_{\et} \big)$ which induces an equivalence of stacks
	\[
		\LocSys(X) \simeq \trun_{\leq 0} \dLocSys(X).
	\]
\end{prop}

\begin{proof}
The proof of \cite[Theorem 4.5.8]{me1} applies.
\end{proof}

\begin{notation}
Let $Z \coloneqq (\cZ, \cO_Z) \in \dAn$ denote a derived $\Ql$-analytic space and $M \in \Mod_{\cO_Z}$. In \cite[\S 5]{porta_rep} it was introduced the analytic square zero extension of $Z$ by $M$ as the derived $\Ql$-analytic space $Z[M] \coloneqq
(\cZ, \cO_Z \oplus M) \in \dAn$, where $\cO_Z \oplus M \coloneqq \Omega^\infty_\an \in \AnRing(\cZ)_{/ \cO_Z}$ denotes the trivial square zero extension of $\cO_Z$ by $M$. In this case, we have a natural composite
	\begin{equation} \label{an:triv_ext}
		\cO_Z \to \cO_Z \oplus M \to \cO_Z
	\end{equation}
in the \infcat $\AnRing(\cZ)_{/ \cO_Z}$ which is naturally equivalent to the identity on $\cO_Z$. We denote $p_{Z, M} \colon \cO_Z \oplus M \to \cO_Z$ the natural projection displayed in \eqref{an:triv_ext}
\end{notation}

\begin{defi} \label{tangent}
Let $Z \in \dAfdl^{\op}$ be a derived $\Q_\ell$-affinoid space. Let $\rho \in \dLocSys(X)(\cO_Z)$ be a continuous representation with values in $\cO_Z$.
The \emph{tangent complex} of $\dLocSys(X)$ at $\rho$ is defined as the fiber
	\[
		\mathbb T_{\dLocSys(X), \rho} \coloneqq \fib_{\rho} \left( p_{\cO_Z} \right)
	\]
where 
	\[ 
		p_{\cO_Z} \colon \dLocSys(X) ( \cO_Z \oplus^{\an} \cO_Z) \to \dLocSys(\cO_Z), 
	\]
is the morphism of stacks
induced from the canonical projection map $p_{\cO_Z, \cO_Z} \colon \cO_Z \oplus \cO_Z \to \cO_Z$.
\end{defi}

The derived stack $ \dLocSys$ is not, in general, representable as derived $\Q_\ell$-analytic stack, as this would entail the representability of its $0$-truncation. Nevertheless we can compute its tangent
complex explicitly:

\begin{lemma}{\cite[Proposition 4.4.9.]{me1}}
Let $\rho \in \dLocSys(X)(\cO_Z)$. We have a natural morphism
	\[
		\mathbb{T}_{ \dLocSys(X), \rho} \to C^*_{\emph{\et}}\left(X, \mathrm{Ad} \left( \rho \right) \right)[1] , 
	\]
which is an equivalence in the derived \infcat $\Mod_{\Gamma(Z)}$.
\end{lemma}

\begin{proof}
The proof of \cite[Proposition 4.4.9]{me1} applies.
\end{proof}

\subsection{The bounded ramification case} In this \S we are going to define a natural derived enhancement of $\abLocSys(X)$ and prove its representability by a derived $\Ql$-analytic stack. Let $X$ be a smooth scheme over an algebraically closed field
$k$ of positive characteristic $p \neq \ell$.

\begin{defi}Consider the sub-site $X^{\tame}_{ \et}$ of the small \'etale site $X_{\et}$ spanned by those \'etale coverings $Y \to X$ satisfying condition (2) in \cref{}. We can form the $\infty$-topos $\Shv^{\tame}(X) \coloneqq \Shv \left( X^{\tame}_{\et} \right)$
of \emph{tamely ramified} \'etale sheaves on the Grothendieck site $X^{\tame}_{\et}$.
\end{defi}

Consider the inclusion of sites $ \iota \colon X^\tame_\et \hookrightarrow X_\et$, it induces a geometric morphism of $\infty$-topoi
	\begin{equation} \label{geo_m}
		g_* \colon \Shv_\et(X) \to \Shv^\tame_\et(X)
	\end{equation}
which is a right adjoint functor to the functor induced by precomposition with $\iota$.

\begin{lemma}
The geometric morphism of $\infty$-topoi $g_* \colon \Shv^\tame_{\emph{\et}}(X) \to \Shv_{\emph{\et}}(X)$ introduced in \eqref{geo_m} is fully faithful.
\end{lemma}

\begin{proof} As the Grothendieck topology on $X_\et^\tame$ is induced by the inclusion functor $\iota \colon X_\et^\tame \to X_\et$, it suffices to prove the corresponding statement for the \infcats of presheaves. More specifically,
the statement of the lemma is a consequence of the assertion that the left adjoint
	\[
		\iota^* \colon \Psh \left( X_{\et} \right) \to \Psh \left( X^{\tame}_{\et} \right),
	\]
given by precomposition along $\iota$,
admits a fully faithful right adjoint. The existence of a right adjoint for $\iota^*$, denoted $\iota_{*}$, follows by the Adjoint functor theorem. The required right adjoint is moreover computed by means of a right Kan extension along $
 \iota$. Let $Y \in X_{\et}^{\tame}$, we can consider $Y \in X_\et$ by means of the inclusion functor $\iota \colon X_\tame^\et \to X_\et$. The comma \infcat $\big( X^\tame_\et \big)_{Y/}$ admits an initial object, namely $Y$ itself. Let
 $\cC_Y \coloneqq \big( X_\et^\tame \big)_{Y / }.$ Given $\cF \in
 \Psh \big( X_\et^\tame \big)$ one can compute 
 	\begin{align*}
		\iota^* \iota_* \cF( Y) & \simeq   \\
		& \simeq \iota_* \cF(Y) \\
		& \simeq \iota^*  \lim_{V \in \cC_Y} \cF(V)  \\
		& \simeq \cF(Y)
	\end{align*}
In particular, the counit of the adjunction $\theta \colon \iota^* \circ \iota_* \to \mathrm{Id}$ is an equivalence. Reasoning formally we deduce that $\iota_*$ is fully faithful and therefore so it is $g_*$.
\end{proof}

\begin{defi}
Let $\Sh^{\tame}(X) \in \pro \left( \cS \right)$ denote the fundamental $\infty$-groupoid associated to the $\infty$-topos $\Shv(X_{\et}^{\tame} )$, which we refer to as the \emph{tame \'etale homotopy type of} $X$.
\end{defi}

\begin{rema} \label{tame_vs_et}
The fact that the geometric morphism $g_* \colon \Shv(X_{\et}^{\tame} ) \to \Shv(X_{\et})$ is fully faithful implies that the canonical morphism 
	\[
		\Sh^{\tame} (X) \to \Sh^{\et}(X)
	\]
induces an equivalence of profinite abelian groups $\pi_i \left( \Sh^{\tame} (X) \right) \simeq \pi_i \left( \Sh^{\et}(X) \right)$ for each $i>1.$
As a consequence one has a fiber sequence
	\[
		\mathrm B \pi_1^w(X) \to \Sh^{\et}(X) \to \Sh^{\tame}(X),
	\]
in the \infcat $\pro(\cS^\fc)$ of profinite spaces.
\end{rema}

\begin{defi}
The derived moduli stack of \emph{wild (pro)-\'etale rank $n$ $\ell$-local systems on $X$} is defined as the functor $\dLocSys^w(X) \colon  \dAfd^{\op} \to \cS$ given informally by the association
	\[
		Z \in \dAfdl^{\op} \mapsto \lim_{n \geq 0} \Map_{\ind(\pro(\cS))} \left( \mathrm B \pi_1^w(X), \mathbf{ \mathrm B \GLn} \left( \tau_{\leq n}\big( \Gamma(Z) \big) \right) \right) \in \cS.
	\]
\end{defi}

\begin{rema}
The functor $ \dLocSys^w(X)$ satisfies descent with respect to the \'etale site $( \dAfd, \tau_{\et})$, thus we can naturally consider $ \dLocSys^w(X)$ as an object of the $\infty$-category of \emph{derived stacks} $\dSt \left( \dAfd , \tau_{\et}, \right)$.
\end{rema}

Suppose now we have a surjective continuous group homomorphism $q \colon \wildpi(X) \to \Gamma$, where $\Gamma$ is a finite group. Such morphism induces a well defined morphism (up to contractible indeterminacy) 
	\[
		\mathrm B q \colon \mathrm B \pi_1^w ( X) \to \mathrm B \Gamma.
	\]
Precomposition along $\mathrm B q$ induces a morphism of derived moduli stacks $\mathrm B q^* \colon \dLocSys(\Gamma) \to \dLocSys^w(X)$. Where $\dLocSys(\Gamma) \colon  \dAfdl \to \cS$ is the functor informally defined by the association
	\[
		Z \in \dAfdl \mapsto \Map_{\ind(\pro(\cS))} \left( \mathrm B \Gamma,  \mathrm B \cEnd(Z)\right) .
	\]	

\begin{rema}
As $\mathrm B \Gamma \in \cS^{\fc} \subseteq \pro \left( \cS^{\fc} \right)$ it follows that, for each $Z \in d \Afdl$, one has a natural equivalence of mapping spaces
	\[
		\Map_{\ind(\pro(\cS))} \left( \mathrm B \Gamma,\rmB \cEnd(Z) \right) \simeq \Map_{\cS} \left( \mathrm B \Gamma,  \mathrm B \GLn(\cO_Z) \right).
	\]
Therefore the moduli stack $\dLocSys \left(\mathrm B \Gamma \right) $ is always representable by a derived $\Q_\ell$-analytic stack which is moreover equivalent to the analytification of the usual (algebraic) \emph{mapping stack}
$\underline{\mathbf{\mathrm{Map}} }\left( \mathrm B \Gamma, \mathrm B \GLn(-) \right) $. The latter is representable by an Artin stack, see \cite[Proposition 19.2.3.3.]{lurieSAG}.
\end{rema}

We can now give a reasonable definition of the moduli of local systems with bounded ramification at infinity:

\begin{defi}
The derived moduli stack of derived \'etale local systems on $X$ wtih \emph{$\Gamma$-bounded ramification at infinity} is defined as the fiber product
	\[
		\abdLocSys (X) : =  \dLocSys(X) \times_{ 	   \dLocSys^w (X)	}  	\dLocSys( \mathrm B \Gamma 	)
	\]
\end{defi}

\begin{prop}
Let $q \colon \pi_1^w(X) \to \Gamma$ be a surjective continuous group homomorphism whose target is finite. Then the $0$-truncation of $\abdLocSys(X) $ is naturally equivalent to $\abLocSys(X) $. In particular, the former is representable by a
$\Q_{\ell}$-analytic stack.
\end{prop}

\begin{proof}
It suffices to prove the statement for the corresponding moduli associated to $\Sh^{\et}(X)$, $\mathrm B \pi_1^w(X)$ and $\mathrm B \Gamma$. Each of these three cases can be dealt as in \cref{123}.
\end{proof}

Similarly to the derived moduli stack $\dLocSys(X)$ we can compute the tangent complex of $\abdLocSys(X)$ explicitly. In order to do so, we will first need some preparations:

\begin{construction} \label{const:mod}
Let $Y \in \pro \left( \cS_{\geq 1}^{\fc} \right)$ be a \emph{$1$-connective} profinite space. Fix moreover a morphism
	\[
		c \colon * \to \cX,
	\]
in the \infcat $\pro \left( \cS^{\fc} \right)$. Notice that such choice is canonical up to contractible indeterminacy due to connectedness of $X$.

Let $\Perf \left( \Q_{\ell} \right)$ the \infcat of perfect $\Q_{\ell}$-modules. One can canonically enhance $\Perf(\Q_{\ell})$ to an object in the \infcat $\cE \Cat$ of $\ind(\pro(\cS))$-enriched \infcats. Consider the full subcategory	
	\[
		\Perf_{\ell} \left( Y \right) \coloneqq \Fun_{\cont} \left( Y, \Perf(\Q_{\ell} ) \right)
	\]
of $\Fun \left( \Mat \left( Y \right) , \Perf( \Q_{\ell} ) \right)$ spanned by those functors $F \colon Y \to \Perf( \Q_{\ell})$ with $M \coloneqq F(*)$ such that the induced morphism
	\begin{equation} \label{eq:cont}
		\Omega \Mat \left(  \cX \right) \to \End \left( M \right)
	\end{equation}
is equivalent to the materialization of a continuous morphism
	\[
		\Omega \cX \to \cEnd \left( M \right)
	\]
in the \infcat $\ind(\pro(\cS))$. Thanks to \cite[Corollary 4.3.23]{me1} the \infcat
$\Perf_\ell(\cX)$ is an idempotent complete stable $\Q_{\ell}$-linear \infcat which
admits a symmetric monoidal structure given by point-wise tensor product.

Consider the \emph{ind-completion} $\Mod_{\Q_{\ell}}(\cX) \coloneqq \ind \left(\Perf_\ell(\cX)  \right)$, which is a presentable stable symmetric monoidal
$\Q_\ell$-linear \infcat, \cite[Corollary 4.3.25]{me1}. We have a canonical functor $p_{\ell} ( \cX ) \colon \Mod_{\Q_\ell} (\cX) \to \Mod_{\Q_\ell}$ given informally by the formula
	\[
		\colim_i F_\in \Mod_{\Q_\ell} \left( Y \right) \mapsto \colim_i \left( F_i(*)  \right) \in \Mod_{\Q_\ell}.
	\]
Given $Z \coloneqq (\cZ, \cO_Z) \in \dAfdl$ a derived $\Q_{\ell}$-affinoid space, we denote
$\Gamma(Z) := \Gamma \left( Z \right)$ the corresponding derived ring of global sections. Consider the extension of scalars \infcat 
	\[
		\Mod_{\Gamma(Z)} \left( Y \right) := \Mod_{\Q_{\ell}} \left( Y \right) \otimes_{\Q_\ell} \Gamma(Z),
	\]
which is a presentable stable symmetric
monoidal $\Gamma(Z)$-linear \infcat, \cite[Corollary 4.3.25]{me1}. We can base change $p_\ell (Y)$ to a well defined  (up to contractible indeterminacy) functor $ p_{\Gamma(Z)} \left( Y \right)  \colon \Mod_{\Gamma(Z)} \left( Y \right) \to \Mod_{\Gamma(Z)}$ given informally by the
association
	\[
		\left( \colim_i F_i \right) \otimes_{\Q_{\ell}} \Gamma(Z) \in \Mod_{\Gamma(Z)} \left( \cX \right) \mapsto \colim_i \left( F_i(*) \otimes_{\Q_\ell} \Gamma(Z) \right) \in \Mod_{\Gamma(Z)}.
	\]
\end{construction}

\begin{prop} \label{tang_comp}
Let $Z \in \dAfd$ be a derived $\Ql$-affinoid space and $\rho \in \abdLocSys( X)(\cO_Z)$. The inclusion morphism of stacks 
	\[
	 	\abdLocSys(X) \hookrightarrow \dLocSys(X)
	\]
induces a natural morphism at the corresponding tangent complexes at $\rho$
	\[
		\mathbb T_{\abdLocSys,  \rho} \to \mathbb T_{\dLocSys,  \rho}
	\]
is an equivalence in the \infcat $\Mod_{\Gamma(Z)}$. In particular, we have an equivalence of $\Gamma(Z)$-modules
	\[
		\mathbb T_{\abdLocSys, \  \rho} \simeq C^*_{\emph{\et}} \left( X , \Ad  \left( \rho  \right)  \right) [1] \in \Mod_{\Gamma(Z)}.
	\]
\end{prop}

\begin{proof}
Let  $\Pi \coloneqq \mathrm{B} q \colon \mathrm B \wildpi ( X) \to \mathrm B \Gamma$ denote the morphism of profinite homotopy types induced from a continuous surjective group homomorphism $q \colon \wildpi(X) \to \Gamma$ whose target is finite.
We can form a fiber sequence
	\begin{equation} \label{fib}
		\cY \to \mathrm B \wildpi(X) \to \mathrm B \Gamma
	\end{equation}
in the $\infty$-category $\pro \left( \cS^{\mathrm{fc}}_{\geq 1} \right)_{*/}$ of pointed $1$-connective profinite spaces. Let $A \coloneqq \Gamma(Z)$ and
consider the \infcats $\Mod_A \left( \Sh^w( X) \right)$ and $\Mod_A \left( \mathrm B \Gamma \right)$ introduced in \cref{const:mod}. 
Let $\cC_{A, n} \left( \mathrm B \pi_1^w(X) \right)$ and $\cC_{A, n} \left( \mathrm B \Gamma \right)$ denote the full subcategories of $\Mod_A \left( \mathrm B \pi_1^w( X) \right)$ and $\Mod_A \left( \mathrm B \Gamma \right)$,
respectively, spanned by modules rank
$n$
free $A$-modules. It is a direct consequence of the definitions that one has an equivalence of spaces 
	\[
		\dLocSys \left( \mathrm B \pi_1^w(X) \right) \simeq \cC_{A, n} \left( \mathrm B \pi_1^w(X) \right)^{\simeq} \text{ and } \dLocSys \left( \mathrm B \Gamma \right) \simeq \cC_{A, n } \left( \mathrm B \Gamma \right)^{\simeq}
	\]
where $(-)^{\simeq}$ denotes the underlying $\infty$-groupoid functor.
The fiber sequence displayed in \eqref{fib} induces an equivalence of $\infty$-categories
	\begin{equation} \label{eq:cats}
		\Mod_A \left( \mathrm B \Gamma \right) \simeq \Mod_A \left( \mathrm B \pi_1^w(X) \right)^\cY
	\end{equation}
where the right hand side of \eqref{eq:cats} denotes the $\infty$-category of $\cY$-equivariant
continuous representations of $\mathrm B \pi_1^w(X)$ with $A$-coefficients. Thanks to \cite[Proposition 4.4.9.]{me1} we have an equivalence of $A$-modules
	\begin{equation} \label{eq:tangSw}
		\mathbb T_{\dLocSys \left( \mathrm B \pi_1^w(X) \right),  \ \rho_{|_{\mathrm B \pi_1^w(X)}}} \simeq \Map_{\Mod_{\Gamma(Z)} \left(\mathrm B \pi_1^w(X)\right) } \left( 1 ,  \rho_{|_{\mathrm B \pi_1^w(X)} }  \otimes \rho_{|_{\mathrm B \pi_1^w(X)}}^{\vee} \right) [1]
	\end{equation}
and similarly,
	\begin{equation}
		\mathbb T_{\dLocSys \left( \mathrm B \Gamma \right), \ \rho_{\Gamma}} \simeq \Map_{\Mod_{\Gamma(Z)} \left(\mathrm B \Gamma \right) } \left( 1 ,  \rho_{\Gamma} \otimes \rho_{\Gamma}^{\vee} \right) [1]
	\end{equation}
By definition of $\rho$, we have an equivalence $\rho^\cY \simeq \rho$, where $(-)^\cY$ denotes (homotopy) fixed points with respect to the morphism $\cY \to \mathrm B \pi_1^w(X)$. Thus we obtain a natural equivalence of $A$-modules:
	\begin{equation} \label{fixed}
		 \Map_{\Mod_{\Gamma(Z)} \left(\mathrm B \pi_1^w(X)\right) } \left( 1 ,  \rho \otimes \rho^{\vee} \right) [1] \simeq \Map_{\Mod_{\Gamma(Z)} \left(\mathrm B \pi_1^w(X) \right) } \left( 1 , (  \rho_{\Gamma} \otimes \rho_{\Gamma}^{\vee} )^\cY \right) [1].
	\end{equation}
Homotopy $\cY$-fixed points are computed by $\cY$-indexed limits.
As the $\cY$-indexed limit computing the right hand side of \eqref{fixed} has identity transition morphisms we conclude that the right hand side of \eqref{fixed} is naturally equivalent to the mapping space
	\begin{equation}
		 \Map_{\Mod_A \left(\mathrm B \pi_1^w(X) \right) } \left( 1 , (  \rho \otimes \rho^{\vee} )^\cY \right) [1] \simeq \Map_{\Mod_A \left(\mathrm B \Gamma \right) } \left( 1 , \Pi_*(  \rho \otimes \rho^{\vee} ) \right) [1]
	\end{equation}
where $\Pi_* \colon \Mod_A \left(\mathrm B \pi_1^w(X)\right)  \to \Mod_A \left( \mathrm B \Gamma \right) $ denotes a right adjoint to the forgetful $\Pi^* \colon \Mod_A \left( \mathrm B \Gamma \right) \to \Mod_A \left( \mathrm B \pi_1^w(X) \right)$. 
As a consequence we have an equivalence
 	\begin{equation}
		 \Map_{\Mod_A \left(\mathrm B \pi_1^w(X)\right) } \left(1 ,  \rho \otimes \rho^{\vee}  \right) [1] \simeq \Map_{ \Mod_A \left(\mathrm B \Gamma \right) } \left( 1 , \Pi_*(  \rho \otimes \rho^{\vee} ) \right) [1]
	\end{equation}
in the \infcat $\cS$. Notice that, by construction
	\begin{equation} \label{eq:gamma}
		\rho_{\Gamma} \otimes \rho_{\Gamma}^{\vee} \simeq \left( \rho \otimes \rho^{\vee} \right)_{\Gamma}
	\end{equation}
in the \infcat $\Mod_A \left( \mathrm B \Gamma \right)$. One has moreover equivalences 
	\begin{equation} \label{eq:comp}
		\Pi_* \left( \rho \otimes \rho^\vee \right)
		\simeq \left( \rho \otimes \rho^\vee \right)_{\Gamma},
	\end{equation}
as the restriction of $\rho \otimes \rho^\vee$ to $\cY$ is trivial. Thanks to \eqref{eq:tangSw} through \eqref{eq:comp}
we conclude that the canonical morphism $\LocSys \left( \mathrm B \Gamma \right) \to \LocSys \left( \mathrm B \pi_1^w(X) \right)$ induces an equivalence on tangent spaces, as desired.
\end{proof}

\begin{construction} \label{const:imp}
Fix a continuous surjective group homomorphism $q \colon \wildpi(X) \to \Gamma$, whose target is finite. Denote by $H$ the kernel of $q$. The profinite group $H$ is an open subgroup of $\wildpi(X)$. For this reason, there exists an open subgroup
$U \leq \pi_1^\et(X)$ such that $U \cap \wildpi(X) = H$. In particular, the subgroup $U$ has finite index in $\pi_1^\et(X)$. As finite \'etale coverings of $X$ are completely determined by finite continuous representations of
$\pi_1^\et(X)$, there exists a finite \'etale covering 
	\[
		f_U \colon Y_U \to X
	\]
such that $\pi_1^\et(X)$ acts on it canonically. Moreover, one has an isomorphism of profinite groups
	\[
		\pi_1^\et(Y) \cong U 
	\]
As a consequence, it follows that $\wildpi(Y_U) \cong H$. Given  $Z \in \Afdl$ and $\rho \in \abdLocSys(X) (\cO_Z)$ it follows by the construction of $f_U \colon
Y_U \to X$ that the restriction 
	\[
		\rho_{
		\vert \Sh^\et(Y)}
	\]
factors through $\Sh^\tame(Y)$. The morphism $f_U \colon Y_U \to X$ induces a morphism of profinite spaces
	\[
		\Sh^\et(Y) \to \Sh^\et(X),
	\]
which on the other hand induces a morphism of stacks $\dLocSys(X) \to \dLocSys(Y_U)$. Moreover, by the above considerations the composite
	\[
		\abdLocSys(X) \to \dLocSys(X) \to \dLocSys(Y),
	\]	
factors through the substack of tamely ramified local systems $\dLocSys \big(\Sh^\tame(Y_U) \big) \hookrightarrow \dLocSys(Y_U)$.
\end{construction}

\begin{lemma} \label{lem:imp}
The canonical restriction morphism of \cref{const:imp}
	\[
		\abdLocSys(X) \to \mathbf{RLocSys}_{\ell, n } (Y_U)
	\]
induces an equivalence
	\[
		\abdLocSys(X) \simeq  \mathbf{RLocSys}_{\ell, n } \big(\Sh^\emphet(Y_U) \big)^{\rmB \Gamma'}
	\]
of stacks.
\end{lemma}

\begin{proof}
By Galois descent, the restriction morphism along $f_U \colon Y_U \to X$ induces an equivalence of stacks
	\[
		\dLocSys(X) \simeq \dLocSys(Y_U)^{\rmB \Gamma'}.
	\]	
Moreover, the considerations of \cref{const:imp} imply that we have a pullback square
	\begin{equation} \label{pull}
	\begin{tikzcd}
		\abdLocSys(X) \ar{r} \ar{d} & \dLocSys(X) \ar{d} \\
		\dLocSys \big( \Sh^\tame(Y_U) \big) \ar{r} & \dLocSys(Y_U)
	\end{tikzcd}
	\end{equation}
in the \infcat $\dSt \big( \dAfdl, \tau_\et \big)$. The result now follows since we can identify \eqref{pull} with
	\[
	\begin{tikzcd}
		\dLocSys \big( \Sh^\tame(Y_U) \big)^{\rmB \Gamma'} \ar{r} \ar{d} & \dLocSys(Y_U)^{\rmB \Gamma'} \ar{d} \\
		\dLocSys \big( \Sh^\tame(Y_U ) \big) \ar{r} & \dLocSys(Y_U)
	\end{tikzcd}
	\]
in the \infcat $\dSt \big( \dAfdl, \tau_\et \big)$. 
\end{proof}

\begin{theorem}
The (derived) moduli stack $\abdLocSys(X)$ is representable by a derived $\Q_{\ell}$-analytic stack.
\end{theorem}

\begin{proof}
Thanks to \cite[Theorem 7.1]{porta_rep} we need to check that the functor $\abdLocSys(X)$ has representable $0$-truncation, it admits a (global) cotangent complex and it is compatible with Postnikov towers. The representability of $t_0(\abdLocSys(X) ) \simeq
\abLocSys(X)$ follows from \cref{main1}. \cref{tang_comp} implies that $\abdLocSys(X)$ admits a global tangent complex. Moreover, by finiteness of $\ell$-adic cohomology for smooth varieties in characteristic $p \neq \ell$, \cite[Theorem 19.1]{milne_et} together
with \cite[Proposition 3.1.7]{me1} for each $\rho \in \abdLocSys(X)(Z)$, the tangent complex at $\rho$
	\[
		\bT_{\abdLocSys(X), \rho} \simeq C^*_{\et} \big( X, \Ad (\rho) \big) [1] \in \Mod_{\Gamma(Z)}
	\]
is a dualizable object of the derived \infcat $\Mod_{\Gamma(Z)}$. Thanks to \cref{lem:imp} we deduce that the existence of a cotangent complex is equivalent to the existence of a global cotangent complex for the derived moduli stack
	\[
		\dLocSys \big( \Sh^\tame(Y_U) \big) \in \dSt \big( \Afdl, \tau_\et \big).
	\]
We are thus reduced to show that $\Sh^\tame (Y) \in \pro ( \cS^\fc)$ is cohomologically perfect and cohomologically compact, see \cite[Definition 4.2.7]{me1} and \cite[Definition 4.3.17]{me1} for the definitions of these notions. As $Y_U$
is a smooth scheme over a field of characteristic $p \neq \ell$, cohomologically perfectness of $\Sh^\tame(Y_U)$ follows by finiteness of \'etale cohomology with $\ell$-adic coefficients, \cite[Theorem 19.1]{milne_et} together with \cite[Proposition 3.1.7]{me1}.
To show that $\Sh^\tame(Y)$ is cohomologically compact we pick a torsion $\bZ_\ell$-module $N$ which can be written as a filtered colimit $N \simeq \colim_{\alpha} N_{\alpha}$ of perfect $\bZ_\ell$-modules. As the tame fundamental group is topologically
of finite type and for each $i > 0 $, the stable homotopy groups $\pi_i \big( \Sh^\tame(Y_U) \big)^{\st}$ are finitely presented the result follows.
For these reasons, the derived moduli stack $\dLocSys \big( \Sh^\tame(Y_U) \big)$ admits a glocal cotangent complex. \cref{lem:imp} implies now that the same is true for $\abdLocSys(X)$. Compatibility with Postnikov towers of $\abdLocSys(X)$ follows from the
fact that the latter moduli is defined as a pullback of stacks compatible with Postnikov towers.
\end{proof}

\section{Comparison statements}

\subsection{Comparison with Mazur's deformation functor}
Let $L$ be a finite extension of $\Q_{\ell}$, $\cO_L$ its ring of integers and $\mathfrak l := \O_L / \mathfrak{m}_L$ its residue field. We denote $\CAlg_{/ \mathfrak l}^{\sm}$ the $\infty$-category of \emph{derived small $k$-algebras} augmented over $
\mathfrak l$.

Let $G$ be a profinite group and $\rho \colon G \to \GLn(L)$ a continuous $\ell$-adic representation of $G$. Up to conjugation, $\rho$ factors through $\GLn(\O_L ) \subseteq \GLn(L)$ and we can consider its corresponding
residual continuous $\mathfrak l$-representation
	\[
		\overline{\rho} \colon G \to \GLn(\fl).
	\]
The representation $\rho$ can the be obtained as the inverse limit of $\{ \overline{\rho}_n \colon G \to \GLn(\cO_L / \mathfrak{m}_L^{n+1}) \}_n$, where each $\overline{\rho}_n \simeq \rho \ \mathrm{mod } \ \mathfrak{m}^{n+1}$. For each $n \geq 0$,
$\overline{\rho}_n $ is a deformation of the residual representation $\overline{\rho}$ to the ring $\cO_L / \mathfrak m_L^{n+1}$.
Therefore, in order to understand continuous representations $\rho \colon G \to \GLn(L)$ one might hope to understand residual representations $\overline{\rho} \colon G \to \GLn(\fl)$ together with their corresponding
deformation theory. For this reason,
it is reasonable to consider the corresponding \emph{derived formal moduli problem}, see \cite[Definition 12.1.3.1]{lurieSAG}, associated to $\overline{\rho}$:
	\[
		\Def_{\overline{\rho}} \colon \CAlg^{\sm}_{/ \fl} \to \cS,
	\]
given informally via the formula
	\begin{equation} \label{deform}
		A \in \CAlg^{\sm}_{/ \fl} \mapsto \Map_{\ind(\pro(\cS))} 	 \left( \mathrm B G, \rmB \cEnd(A) \right) \times_{ \Map_{\ind(\pro(\cS))} \left( \mathrm B G,
		\rmB \cEnd(A) \right) } \{ \overline{\rho} \} \in \cS
		.
	\end{equation}

\begin{construction}
\cite[Proposition 4.2.6]{me1} and its proof imply that one has an equivalence between the tangent complex of $\Def_{\overline{\rho}}$ and the complex of continuous cochains of $\Ad ( \overline \rho)$
	\begin{equation} \label{cont_coh}
		\mathbb T_{\Def_{\overline{\rho}}} \simeq C^*_{\mathrm{cont}} \left( G, \Ad( \rho ) \right)[1]
	\end{equation}
in the $\infty$-category $\Mod_\fl$. 
Replacing $\mathrm B G$ in \eqref{deform} by \'etale homotopy type of $X$, $\Sh^{\et}(X)$, and $C^*_{\cont}$ by $C^*_{\et}$ in \eqref{cont_coh} it follows by \cite[Theorem 19.1]{milne_et} together with \cite[Theorem 6.2.5]{lurieDAGXII} that $
\Def_{\overline{\rho}}$ is \emph{pro-representable} by
a local Noetherian derived ring $A_{\overline{\rho}} \in \CAlg_{/ \fl}$ whose residue field is equivalent to $\fl$. Moreover, $A_{\overline{\rho}}$ is complete with respect to the augmentation ideal $\mathfrak{m}_{A_{\overline{\rho}}}$
(defined as the kernel of the homomorphism $\pi_0 \left(  A_{\overline{\rho}} \right) \to k$ of ordinary rings).
It follows that $A_{\overline{\rho}}$ admits a natural structure of a derived $W(\fl)$-algebra, where $W(\fl)$ denotes the ring of Witt vector of $\fl$.
As $\overline{\rho}$ admits deformations to $\cO_L$, for e.g. $\rho$ itself, we have that $\ell \neq 0 $ in $\pi_0(A_{\overline{\rho}})$.
\end{construction}

\begin{notation}
Denote by $L^\unr \coloneqq \Frac \left( W(\fl) \right)$ the field of fractions of $W(\fl)$. It corresponds to the maximal unramified extension of $\Q_\ell$ contained in $L$.
\end{notation}

\begin{prop}
Let $\mathrm t_{\leq 0} \left( \Def_{\overline{\rho}} \right)$ denote the $0$-truncation of the derived formal moduli problem $\Def_{\overline{\rho}}$, i.e. the restriction of $\Def_{\overline{\rho}}$ to the full subcategory of ordinary Artinian rings augmented over $\fl$,
$\CAlg_{/ \fl }^{\mathrm{sm},\heartsuit} \subseteq \CAlg_{/ \fl}^{\sm}$.
Then $\mathrm t_{\leq 0} \left( \Def_{\overline{\rho}} \right)$ is equivalent to Mazur's deformation functor introduced in \cite[Section 1.2]{mazurDG} and $\pi_0(A_{\overline{\rho}})$
is equivalent to Mazur's universal deformation ring.
\end{prop}

\begin{proof}
Given $R \in \CAlg_{/ \fl}^{\sm, \heartsuit} \subseteq \CAlg_{/ \fl }^{\sm}$ an ordinary (Artinian) local $\fl$-algebra,
the object $\rmB \cEnd(R) \in  \ind(\pro(\cS))$ is \emph{$1$-truncated}. Therefore one has a natural equivalence of spaces
	\begin{equation} \label{eq:0}
		\mathrm{t_0} \left( \Def_{\overline{\rho}} \right) (R) \simeq \Map_{\ind(\pro(\cS)) } \left( \mathrm B \pi_1^{\et}(X), \rmB \cEnd(A) \right) \times_{\Def_{\overline{\rho}}(k)} \{ \overline{\rho} \}.
	\end{equation}
By construction, the ordinary $W(\fl)$-algebra $\pi_0(A_{\overline{\rho}})$ pro-represents the functor $\mathrm t_0 \left( \Def_{\overline{\rho}} \right) \colon \CAlg_{/\fl}^{\sm, \heartsuit} \to \cS
$. As a consequence, the
mapping space on the right hand side of \eqref{eq:0} is $0$-truncated and the set of $R$-points corresponds to deformations of $\overline{\rho}$ valued in $R$. This is precisely Mazur's deformation functor, as introduced in \cite[Section 1.2]{mazurDG}, concluding
the proof.
\end{proof}

\subsection{Comparison with S. Galatius, A. Venkatesh derived deformation ring} In the case where $X $ corresponds to the spectrum of a maximal unramified extension, outside a finite set $S $ of primes, of a number field $L$ and $\rho \colon G_X \to \GLn(K)$
is a continuous representation, the corresponding derived $W(k)$-algebra was first introduced and extensively studied in \cite{galatius_dg}.

\subsection{Comparison with G. Chenevier moduli of pseudo-representations} In this section we will compare our derived moduli stack $\dLocSys(X)$ with the construction of the moduli of \emph{pseudo-representations} introduced in \cite{chenevier}.
We prove that $\dLocSys(X)$ admits an admissible analytic substack which is a disjoint union of the various $\Def_{\overline{\rho}}$.
Such disjoint union of deformation functors admits a canonical map to the moduli of pseudo-representations of introduced in \cite{chenevier}. Such morphism of derived stacks is obtained as the composite of the $0$-truncation functor 
followed by the morphism which associates to a continuous representation $\rho $ its corresponding pseudo-representation, see \cite[Definition 1.5]{chenevier}. Nevertheless, the derived moduli stack $\dLocSys(X)$
has more points in general, and we will provide a typical example in order to illustrate this phenomena.

\begin{prop} \label{moc}
Let $\overline{\rho} \colon \pi_1^{\emph{\et}}(X) \to \GLn(\overline{\mathbb F}_\ell)$ be a continuous residual $\ell$-adic representation.
To $\overline{\rho}$ we can attach a derived $\Q_\ell$-analytic space $\Def_{\overline{ \rho}}^{\rig } \in \dAn_{\Q_\ell}$ for which every closed point $\rho \colon \Sp L \to \Def_{\overline{\rho}}^\rig$ is equivalent to a continuous deformation of $\overline{\rho}$ over
$L$.
\end{prop}

\begin{proof}
Denote by $\dfSch_{W(\fl)}$ the \infcat of \emph{derived formal schemes} over $W(\fl)$, introduced in \cite[section 2.8]{lurieSAG}.
The local Noetherian derived $W(\fl)$-algebra $A_{\overline{\rho}}$ is complete with respect to its maximal ideal $\mathfrak m_{A_{\overline{\rho}}}$. For this reason, we can consider its associated derived formal scheme $\Spf A_{\overline{\rho}} \in 
\dfSch_{W(\fl)}$.

Let $A \in \CAlg^{}_{W(\fl)}$ denote an admissible derived $W(\fl)$-algebra, see \cite[Definition 3.1.1]{me2}. We have an equivalence of mapping spaces
	\[
		\Map_{\dfSch_{W(\fl)}} \left( \Spf A, \Spf A_{\overline{\rho}} \right) \simeq \Map_{\CAlg_{W(\fl)}^{\ad}} \left( A_{\overline{\rho}}, A \right).
	\]
Notice that as $A$ is a $\ell$-complete topological almost of finite type over $W(k)$, the image of each $t \in \mathfrak m_{A_{\overline{\rho}}}$ is necessarily a topological nilpotent element of the ordinary commutative ring $\pi_0(A)$.
Let $\mathfrak m \subseteq \pi_0(A)$ denote a maximal ideal of $\pi_0(A)$ and let $\left( A\right)^{\wedge}_{\mathfrak{m}}$ denote the $\mathfrak m$-completion of $A$.
There exists a faithfully flat morphism of derived adic $W(k)$-algebra
	\[
		A \to A' \coloneqq \prod_{\mathfrak m \subseteq \pi_0(A)} 	\left( A \right)^\wedge_{\mathfrak{m}}
	\]
where the product is labeled by the set of maximal ideals of $\pi_0(A)$. By fppf descent we have an equivalence of mapping spaces
	\begin{equation} \label{ffdes}
		 \Map_{\CAlg^{\ad}_{W(k)}} \left( A_{\overline{\rho}}, A \right) \simeq \lim_{[n] \in \mathbf \Delta^{\op}} \Map_{\CAlg_{W(k)}^{\ad}} \left( A_{\overline{\rho}}, A'_{[n]} \right) 
	\end{equation}
where $A'_{[n]} \coloneqq A' \widehat{\otimes}_A \dots \widehat{\otimes}_A A'$ denotes the $n+1$-tensor fold of $A'$ with itself over $A$ computed in the \infcat of derived adic $W(k)$-algebras $\CAlg_{W(k)}^\ad$.
For a fixed $[n] \in \mathbf \Delta^{\op}$ we an equivalence of spaces
	\[
		\Map_{\CAlg^{\ad}_{W(k)} } \left( A_{\overline{\rho}}, A'_{[n]} \right) \simeq
		  \Def_{\overline{\rho}} \left( A'_{[n]} \right).
	\]
For each $[n] \in \mathbf \Delta^{\op}$ we obtain thus a natural inclusion morphism $\theta_{[n]} \colon \Map_{\CAlg^{\ad}_{W(k)}} \left( A_{\overline{\rho}}, A'_{[n]} \right) \to \dLocSys(X)(A'_{[n]})$.  The $\theta_{[n]}$ assemble together and by fppf descent induce
a morphism $\theta \colon \Map_{\CAlg^{\ad}_{W(k)}} \left( A_{\overline{\rho}}, A \right) \to \dLocSys(X) (A)$. By construction, $\theta$ induces a natural map of mapping spaces	
	\[
		\Map_{\CAlg^{\ad}_{W(k)}} \left( A_{\overline{\rho}}, A \right) \to \prod_{\mathfrak m \subseteq \pi_0(A)} \bigg( \dLocSys(X)(A) \times_{\Def_{\overline{\rho}} \left( A^\wedge_{\mathfrak m} \right)} \dLocSys(X)(A^\wedge_{\mathfrak m}) 
		\bigg)
	\]
which is equivalence of spaces.
In order words $\Spf A_{\overline{\rho}}$ represents the moduli functor which assigns to each affine derived formal scheme
$\Spf A$, over $W(\fl)$, the space of continuous representations $ \rho \colon \Sh^\et(X) \to \mathrm B \GLn(A)$ such that for each maximal ideal $\mathfrak m \subseteq \pi_0(A)$ the induced representation 
	\[
		\left( \rho \right)^\wedge_{\mathfrak{m}} \colon \Sh^\et(X) \to \mathrm B \GLn \left( \left( A^\wedge_{\mathfrak{m}} \right) \right)
	\]
is a deformation of $\overline{\rho} \colon \Sh^{\et}(X) \to \mathrm B \GLn(k)$. The formal spectrum $\Spf A_{\overline{\rho}}$ is locally admissible, see \cite[Definition 3.1.1]{me2}. We can thus consider its rigidificiation introduced in
\cite[Proposition 3.1.2]{me2} which we denote by $\Def^{\rig}_{\overline{\rho}}
\coloneqq \left( \Spf A_{\overline{\rho}} \right)^\rig \in \dAnl$. Notice that
$\Def^{\rig}_{\overline{\rho}}$ is not necessarily derived affinoid. 

Let $Z \in \dAfdl$, \cite[Corollary 4.4.13]{me2} implies that any given morphism $ f \colon Z \to \left( \Spf A_{\overline{\rho}} \right)^{\rig}$ in $\dAnl$ admits necessarily a formal model, i.e., it is equivalent to the rigidification of a morphism
	\[
		\mathfrak f \colon \Spf A \to \Spf A_{\overline{\rho}},
	\]
where $A  \in \CAlg_{W(k)}^{\ad}$ is a suitable admissible derived $W(\fl)$-algebra. The proof now follows from our previous discussion.
\end{proof}

The proof of \cref{moc} provides us with a canonical morphism of derived moduli stacks $\Def^{\rig}_{\overline{\rho}} \to \LocSys(X)$. Therefore, passing to the colimit over all continuous representations
	\[
		\overline{\rho} \colon \pi_1^{\et}(X) \to \GLn(\mathbb F_\ell)
	\]
provides us with a morphism 
	\begin{equation} \label{Psi}
		\theta   \colon
															\coprod_{\overline{\rho} } \Def_{\overline{\rho} }^\rig
																\to
																															\dLocSys(X)							
	\end{equation}
in the $\infty$-category $\dSt(\dAfdl, \tau_{\et})$.

\begin{prop} \label{open_im}
The morphism of derived $\Q_\ell$-analytic stacks 
	\[
		\theta \colon \coprod_{\rho \colon \pi_1^\emphet(X) \to \GLn(\bar{\Q}_{\ell}) } \Def_{\rho}^\rig \to \LocSys(G)
	\]
displayed in \eqref{Psi} exhibits the left hand side as an analytic subdomain of the right hand side.
\end{prop}

\begin{proof}
Let $\overline{\rho} \colon \pi_1^\et(X) \to \GLn( \bF_\ell)$ be a continuous representation. The induced morphism
	\[
		\theta_{\overline{\rho}} \colon \Def^\rig_{\overline{\rho}} \to \dLocSys(X)
	\]
is an \'etale morphism of derived stacks, which follows by noticing
that $\theta_{\overline{\rho}}$ induces an equivalence at the level of tangent complexes. Moreover, \cref{moc} implies that $ \theta_{\overline{\rho}} \colon \Def^\rig_{\overline{\rho}} \to \dLocSys(X)$ exhibits the former as a substack of the latter.
It then follows that the morphism is locally an admissible subdomain inclusion. The result now follows.
\end{proof} 

\cref{open_im} implies that $\dLocSys(X)$ admits as an analytic subdomain the disjoint union of those derived $\Q_{\ell}$-analytic spaces $\Def_{\overline{\rho}}^{\rig}$.
One could then ask if $\theta$ is itself an epimorphism of stacks and thus an equivalence of such. However,
this is not the case in general as the following example illustrates:

\begin{exem} \label{ex_surj}
Let $G = \mathbb Z_{\ell}$ with its additive structure and let $A = \Q_{\ell} \langle T \rangle$ be the (classical) Tate $\Q_{\ell}$-algebra on one generator. Consider the following continuous representation
	\[
		\rho \colon G \to \GL_2 (\Q_{\ell} \langle T \rangle),
	\]
given by
	\[
		1 \mapsto 
		\begin{bmatrix}
			1 & T \\
			0 & 1
		\end{bmatrix}.
	\]
It follows that $\rho$ is a $\Q_{\ell} \langle T \rangle$-point of $\LocSys( \bZ_{\ell})$ but it does not belongs to the image of the disjoint union $\Def^{\rig}_{\overline{\rho}}$ as $\rho$ cannot be factored as a point belonging to the interior of the closed unit disk
$\mathrm{Sp} \left( 	\Q_{\ell} \langle T \rangle 		\right)$.
\end{exem}

\begin{rema}
As \cref{ex_surj} suggests, when $n =2$ the derived moduli stack $\dLocSys(X)$ does admit more points than those that come from deformations of its closed points.
However, we do not know if $\dLocSys$ can be written as a disjoint union of the closures of $\Def^{\rig}_{\overline{\rho}}$ in $\LocSys(X)$. However, when $n = 1$ the analytic subdomain morphism $\theta$ is an equivalence in the \infcat $\dSt \big( \dAfdl,
\tau_\et \big)$.
\end{rema}

\section{Shifted symplectic structure on $\dLocSys(X)$}

Let $X$ be a smooth and proper scheme over an algebraically closed field of positive characteristic $p>0$. Poincar\'e duality provide us with a canonical map
	\[
		\varphi \colon C^*_{\et} \left( X , \Q_\ell \right) \otimes_{\Q_\ell}  C^*_{\et} \left( X , \Q_\ell \right)  \to \Q_\ell[- 2 d]
	\]
in the derived \infcat $\Mod_{\Q_\ell}$ is non-degenerate, i.e., it induces an equivalence of \emph{derived} $\Q_\ell$-modules
	\begin{equation} \label{pd}
		 C^*_{\et} \left( X , \Q_\ell \right) \to  C^*_{\et} \left( X , \Q_\ell \right) ^\vee [-2d],
	\end{equation}
in $\Mod_{\Q_\ell}$. As we have seen in the previous section, we can identify the left hand side of \eqref{pd} with a (shit) of the tangent space of $\dLocSys(X)$ at the trivial representation. Moreover, the equivalence \label{pd} holds if we consider \'etale (co)chains
with more general coefficients. The case that interest us is taking \'etale cohomology with $\Ad( \rho)$-coefficients for a continuous representation $\rho \colon \pi_1^{\et}(X) \to \GLn(A)$, with $A \in \Afdl$. Let $\rho \in \dLocSys(X)(Z)$, we can regard $\rho$
as a dualizable object of the symmetric monoidal \infcat $\Perf^{\ad}_\ell(X) \coloneqq \Fun_{\cE \Cat} \big( \Sh^\et(X), \Perf( A ) \big)$. Let $\rho^\vee$ denote a dual for $\rho $. By definition of dualizable objects, we have a canonical trace map
	\[
		\tr_\rho \colon \rho \otimes \rho^\vee \to 1_{\Perf^{\ad}_\ell(X)}
	\]
in the \infcat $\Perf^{\ad}_\ell(X)$ and $1_{\Perf^{\ad}_\ell(X)}$ denotes the unit object of the latter \infcat. Therefore, passing to mapping spaces, we obtain a natural composite
	\begin{align}
		\Map_{\Perf^{\ad}_\ell(X)} \left( 1, \Ad(\rho) \right) \otimes \Map_{\Perf^{\ad}_\ell(X)} \left( 1, \Ad(\rho) \right) & \xrightarrow[]{\mathrm{mult}}
		 \Map_{\Perf^{\ad}_\ell(X)} \left( 1, \Ad(\rho)  \right) \\
		 & \xrightarrow[]{\tr_\rho} \Map_{\Perf^{\ad}_\ell(X)} \left( 1, 1 \right)
	\end{align}
in the \infcat $\Mod_{\Gamma(Z)}$.
By identifying the above with \'etale cohomology coefficients with coefficients we obtain a non-degenerate bilinear form
	\begin{equation} \label{PDet}
		C^*_\et \big( X, \Ad(\rho) \big)[1] \otimes C_\et^*(X, \Ad(\rho) \big)[1] \to C_\et^* \big( X, \Ad(\rho ) \big)[2] \xrightarrow{\tr_\rho} C^*_\et \big(X,  \Gamma(Z) \big) [2 -2d]
	\end{equation}
in the \infcat $\Mod_{\Gamma(Z)}$. Moreover, this non-degenerate bilinear form can be interpreted as a Poincar\'e duality statement with $\Ad(\rho)$-coefficients.

Our goal in this \S is to construct a shifted symplectic form $\omega$ on $\dLocSys(X)$ in such a way that its underlying bilinear form coincides precisely with the composite \eqref{PDet}. We will also analyze some of its consequences. Before continuing our
treatment we will state a $\Ql$-analytic version of the derived HKR theorem, first proved in the context of derived algebraica geometry in \cite{toen_s1}.

\begin{theorem}[Analytic HKR Theorem] \label{analytic_HKR}
Let $k$ denote either the field of complex numbers or a non-archimedean field of characteristic $0$ with a non-trivial valuation. Let $X  \in \dAn_k$ be a derived $k$-analytic space. Then there is an equivalence of derived analytic spaces
	\[
		X \times_{X \times X } X \simeq \rmT X[-1],
	\]
compatible with the projection to $X$.
\end{theorem}

The proof of \cref{analytic_HKR} is a work in progress together with F. Petit and M. Porta, which the author hopes to include in his PhD thesis.

\subsection{Shifted symplectic structures} In this \S we fix $X$ a smooth scheme over an algebraically closed field $k$ of positive characteristic $p $.

In \cite{toen_ss} the author proved the existence of shifted symplectic structures on certain derived algebraic stacks which cannot be presented as certain mapping stacks. As $\dLocSys(X)$ cannot be presented as usual analytic mapping stack,
we will need to apply the results of \cite{toen_ss} to construct the desired shifted sympletic structure on $\dLocSys(X)$.

\begin{defi}
Consider the canonical inclusion functor $\iota \colon \dSt \left( \dAfdl, \tau_{\et}, P_{\sm} \right) \subseteq \Fun \left( \dAfdl, \cS \right)$. The functor $\iota$ admits a left adjoint which we refer to as \emph{the stackification functor} $\left(- \right)^{\mathrm{st}} \colon
\Fun \left( \dAfdl, \cS \right) 
\to \dSt \left( \dAfdl, \tau_{\et}, P_{\sm} \right)$.
\end{defi}

\begin{defi}
Consider the functor $ \PerfSys^f \colon \dAfdl \to \cS$ which is defined via the assignment
	\[
		Z \in \dAfdl^{\op} \mapsto \Map_{\cE \Cat}  \left( \Sh^{\et}(X), \Perf \big( \Gamma(Z) \big) \right) \in \cS
	\]
where we designate $\Perf \big(\Gamma(Z) \big)
$ to be the  $\ind(\pro(\cS))$-enriched \infcat of perfect $\Gamma(Z)$-modules, which is equivalent to the subcategory of dualizable objects in the \infcat of Tate modules on $\Gamma(Z)$, $\Mod^{\mathrm{Tate}}_{\Gamma(Z)}$, \cite{need reference 
here}. We define the moduli stack $\PerfSys \in \dSt \left( \dAfdl, \tau_{\et}, \right)$ as the stackyfication of $\PerfSys^f$.
\end{defi}

\begin{rema}
This is an example of a moduli stack which cannot be presented as a usual mapping stack, instead one should think of it as an example of a \emph{continuous
mapping stack}.
\end{rema}

\begin{notation}
We will denote $\Cat^{\otimes}$ the \infcat of (small) symmetric monoidal \infcats.
\end{notation}

\begin{defi}
Let $\cC \in \Cat^\otimes$ be a symmetric monoidal \infcat. We say that $\cC$ is a rigid symmetric monoidal \infcat if every object $C \in \cC$ is dualizable.
\end{defi}

\begin{notation} \label{rigCat}
We denote by $\rigCat$ the \infcat of small rigid symmetric monoidal \infcats.
\end{notation}

 Consider the usual inclusion of \infcats $\cS \hookrightarrow \Cat$, it admits a right adjoint, denoted
 	\[
		(-)^{\simeq} \colon \Cat \to \cS
	\]
which we refer as the \emph{underlying $\infty$-groupoid functor}. Given $\cC \in \Cat$ its underlying $\infty$-groupoid $\cC^{\simeq} \in \cS$ consists of the maximal subgroupoid of $\cC$, 
i.e., the subcategory spanned by equivalences in $\cC$.

\begin{lemma} \label{lem:rigidity}
There exists a valued $\Cat^{\st, \omega, \otimes}$-valued pre-sheaf 
	\[
		\Perf^\ad_\ell(X) \colon \dAfdl \to \Cat
	\]
given on objects by the formula
	\[
		Z \in \dAfdl \mapsto \Fun_{\cE \Cat} \left( X, \Perf \big( \Gamma(Z) \big) \right).	
	\]
Moreover, the underlying derived stack $(-)^{\simeq} \circ \Perf^\ad_\ell(X) \in \dSt \big( \Afdl, \tau_\et \big)$ is naturally
equivalent to derived stack $\PerfSys \in \dSt \left( \dAfdl, \tau_{\emphet}\right)$.
\end{lemma}

\begin{proof}
The construction of $\Perf^\ad_\ell(X)$ is already provided in \cite[Definition 4.3.11]{me1}. Moreover, it follows directly from the definitions that $ \big( \Perf^\ad_\ell(X) \big)^\simeq \simeq \PerfSys(X)$.
\end{proof}

\cref{lem:rigidity} is useful because it place us in the situation of \cite[\S 3]{toen_ss}. Therefore, we can run the main argument presented in \cite[\S 3]{toen_ss}. Before doing so, we will need to introduce some more ingredients:

\begin{defi}
Let $H \left( \Perf^\ad_\ell(X) \right) \colon \dAfdl^{\op} \to \cS$ denotes the sheaf defined on objects via the formula
	\[
		Z \in \dAfdl^{\op} \mapsto \Map_{\Perf^\ad_\ell(X) \big( \Gamma(Z) \big)} \left( 1, 1 \right) \in \cS,
	\]
where $\mathbf 1 \in \Perf^\ad_\ell(X) (\Gamma(Z))$ denotes the unit of the corresponding symmetric monoidal structure on $\Perf^\ad_\ell(\Gamma(Z))$.
\end{defi}

\begin{defi}
Let $\cO \colon \dAfdl^{\op} \to \CAlg_{\Q_\ell}$ denote the sheaf on $(\Afdl, \tau_\et)$ given on objects by the formula
	\[
		Z \in \dAfdl^{\op} \mapsto  \Gamma \left( Z \right) \in \CAlg_{\Q_\ell}.
	\]
\end{defi}

\begin{construction} \label{const:pair}
One is able to define a \emph{pre-orientation}, in the sense of \cite[Definition 3.3]{toen_ss}, on the $\rigCat$-value stack $\Perf^\ad_\ell(X)$
	\[
		\theta \colon H \left( \Perf^\ad_{\ell}(X)\right) \to \cO[-2d],
	\]
as follows: let $Z \in \dAfdl$ be a derived $\Q_\ell$-affinoid space. We have a canonical equivalence in the \infcat $\Mod_{\Gamma(Z)}$
	\begin{equation} \label{eq:ii}
		 \beta_{\Gamma(Z)} \colon \Map_{\Perf^\ad_\ell(X)(\Gamma(Z))} \left( 1, 1 \right) \simeq C^*_{\text{\et}} \left( X, \Gamma(Z) \right),
	\end{equation}
by the very construction of $\Perf^\ad_\ell \big( \Gamma(Z) \big)$. Moreover, the projection formula for \'etale cohomology produces a canonical equivalence
	\[
		C^*_{\text{et}} \left( X, \Gamma(Z) \right) \simeq C^*_{\text{et}} \left( X, \Q_\ell \right) \otimes_{\Q_\ell} \Gamma(Z)
	\]
in the \infcat $\Mod_{\Q_\ell}$. As $X$ is a connected smooth scheme of dimension $d$ over an algebraically closed field we have a canonical map on cohomology groups
	\[
		\alpha \colon \Q_\ell \simeq H^0 \left( X_{\et}, \Q_\ell \right) \otimes H^{2d} \left( X_{\et}, \Q_\ell \right) \to \Q_\ell 
	\]
which is induced by Poincar\'e duality. Consequently, the morphism $\alpha$ induces, up to contractible indeterminacy, a canonical morphism
	\begin{equation} \label{eq:iii}
		C^*_{\et}(X, \Q_\ell) \to \Q_\ell [-2d].
	\end{equation}
in the \infcat $\Mod_{\Ql}$.
\eqref{eq:ii} together with base change of \eqref{eq:iii} along the morphism $\Q_\ell \to \Gamma(Z) $ provides us with a natural morphism
	\[
	 	\Map_{\Perf^\ad_\ell(X) (\Gamma(Z)) } \left( 1, 1 \right) \to \Gamma(Z) [-2d].
	\]
By naturality of the previous constructions, we obtain a morphism pre-orientation
	\[
		\theta \colon H \left( \Perf^\ad_\ell(X) \right) \to \cO[-2d],
	\]
which corresponds to the desired orientation.

Given $Z \in \dAfdl$, the \infcat $\Perf^\ad_\ell \big( \Gamma(Z) \big)$ is rigid. Thus for a given object $\rho \in \Perf^\ad_\ell \big( \Gamma(Z) \big)$ we have a canonical trace map
	\[
		\mathrm{tr}_{\rho} \colon \Ad \left( \rho \right) \to 1_{		} .
	\]
which together with the symmetric monoidal structure provide us with a composite of the form
	\begin{align}
		 \Map_{\Perf^\ad_\ell(X)(\Gamma(Z))} \left( 1 , \Ad(\rho) \right) \otimes \Map_{\Perf^\ad_\ell(X)(\Gamma(Z))} \left(  1, \Ad( \rho) \right) \to & \Map_{\Perf^\ad_\ell(X)(\Gamma(Z))} \left( 1 , \Ad(\rho)  \otimes \Ad (\rho) 	
		 \right)  \\ 
		\to \Map_{\Perf^\ad_\ell(X)(\Gamma(Z))} \left( 1, \Ad(\rho) \right)	 \to & \Map_{\Perf^\ad_\ell(X)(\Gamma(Z)} \left( 1, 1 \right) \to  \Gamma(Z)[-2d] &	
	\end{align}
which we can right equivalently as a morphism
	\[
		C^*_{\et} \left( X, \Ad(\rho) \right) \otimes C^*_{\et} \left( X, \Ad(\rho) \right) \to \Gamma(Z)[2-2d],
	\]
which by our construction coincides with the base change along $\Q_\ell \to \Gamma(Z)$ of the usual \emph{pairing} given by \emph{Poincar\'e Duality}.
\end{construction}

\begin{lemma}
Let $Z \in \dAfdl$ be a derived $\Ql$-affinoid space. The pairing of \cref{const:pair}
	\[
		\Map_{\Perf^\ad_\ell(X)(\Gamma(Z)} \left( 1, \Ad (\rho) \right) \otimes \Map_{\Perf^\ad_\ell(X)(\Gamma(Z))} \left( 1, \Ad(\rho) \right) \to \Gamma(Z) [-2d]
	\]
is non-degenerate. In particular, the pre-orientation $\theta \colon H \left( \Perf^\ad_\ell(X) \right) \to \cO[-2d]$ is an orientation, see \cite[Definition 3.4]{toen_ss} for the latter notion.
\end{lemma}

\begin{proof}
Let $\rho \in \PerfSys(X) ( \cO_Z)$ be an arbitrary continuous representation with $\cO_Z$-coefficients. We wish to prove that the natural mapping
	\[
		\Map_{\Perf^\ad_\ell(X)(\Gamma(Z)} \left( 1, \Ad (\rho) \right) \otimes \Map_{\Perf^\ad_\ell(X)(\Gamma(Z))} \left( 1, \Ad(\rho) \right) \to \Gamma(Z) [-2d]
	\]
is non-degenerate. As $Z $ lives over $\Ql$ and $p \neq \ell$ it follows that $\rho \in \mathbf{PerfSys}_{\ell, \Gamma}(X)$
for a sufficiently large finite quotient $q \colon \wildpi(X) \to \Gamma$. It then follows by \cite[Proposition 4.3.19]{me1} together with \cref{lem:imp} that
$\rho$ can be realized as the $\rmB \Gamma$-fixed points of a given $\widetilde{\rho} \colon \Sh^\tame(Y) \to \rmB \GLn(A_0)$, where $Y \to X$ is a suitable \'etale covering and  $A_0 \in \adCAlg$ is an admissible derived $\bZ_\ell$-algebra such that
	\[
		\big( \Spf A_0 \big)^\rig \simeq Z, 
	\]
in the \infcat $\dAfdl$. We notice that it suffices then to show the statement for the residual representation $\rho_0 \colon \Sh^\tame (Y) \to \rmB \GLn(A_0 / \ell)$, where $A_0 / \ell$ denotes the pushout
	\[
	\begin{tikzcd}
		A_0[t] \ar{r}{t \mapsto \ell}  \ar{d}{t \mapsto 0} & A_0 \ar{d} \\
		A_0 \ar{r} & A_0 / \ell
	\end{tikzcd}
	\]
computed in the \infcat $\adCAlg$. We can write $A_0 / \ell$ as a filtered colimit of free $\bF_\ell$-algebras $\bF_\ell[T_0, \dots, T_m]$, where the $T_i $ sit in homological degree $0$. As $\Sh^\tame(Y)$ is cohomological compact we reduce
ourselves to prove the statement by replacing $\rho_0$ with a continuous representation with values in some polynomial algebra $\bF_\ell[T_0, \dots, T_m]$. The latter is a flat module over $\bF_\ell$. Therefore, thanks to Lazard's theorem
\cite[Theorem 8.2.2.15]{lurieHA} we can further reduce ourselves to the case where $\rho_0$ is valued in a finite $\bF_\ell$-module. The result now follows by the \cref{const:pair} together with the projection formula for \'etale cohomology
and Poincar\'e duality for \'etale cohomology.
\end{proof}

As a corollary of \cite[Theorem 3.7]{toen_ss} one obtains the following important result:

\begin{theorem}
The derived moduli stack $\PerfSys(X) \in \dSt \left( \dAfdl, \tau_{\et} \right)$ admits a canonical shifted symplectic structure $\omega \in \rmH \rmC \big(\PerfSys(X) \big)$, where the latter denotes cyclic homology of the derived moduli stack
$\PerfSys(X)$. Moreover, given $Z \in \dAfdl$ and $\rho \in \PerfSys \big( \Gamma(Z) \big)$, the shifted symplectic structure $\omega$ on $\PerfSys(X)$
is induced by \emph{\'etale Poincar\'e duality}
	\[
		C^*_{\emph{\et}}\left(X, \Ad(\rho) \right) [1] \otimes C^*_{\emph{\et}} \left( X, \Ad(\rho) \right) [1] \to \Gamma(Z)[2-2d].
	\]
\end{theorem}

\begin{proof}
This is a direct consequence of our previous discussion together with the argument used in \cite[Theorem 3.7]{toen_ss}.
\end{proof}

\subsection{Applications} Consider the canonical inclusion $\iota \colon \dLocSys(X) \hookrightarrow \PerfSys(X)$. Pullback along the morphism $\iota$ on cyclic homology induces a well defined, up to contractible indeterminacy, morphism
	\[
		\iota^* \colon \rmH \rmC \big( \PerfSys(X) \big) \to \rmH \rmC \big( \dLocSys(X) \big).
	\]
We then obtain a canonical closed form $\iota^* (\omega ) \in \rmH \rmC \big( \dLocSys(X) \big)$. Moreover, as $\iota$ induces an equivalence on tangent complexes, the closed form $\iota^* (\omega ) \in \rmH \rmC \big( \dLocSys(X) \big)$ is non-degenerate,
thus a $2-2d$-shifted symplectic form. Similarly, given a finite quotient $q \colon \wildpi(X) \to \Gamma$, we obtain a $2-2d$-shifted symplectic form on the derived $\Ql$-analytic stack $\abdLocSys(X)$. The existence of the sifted symplectic form entails
the following interesting result:

\begin{defi}
Let $\bL_{\dLocSys(X)}$ denote the cotangent complex of the derived moduli stack $\dLocSys(X)$. We will denote by 
	\[
		C^*_{\dR} \big( \dLocSys(X) \big) \coloneqq \Sym^* \big( \bL_{\dLocSys(X)} \big) \in \Coh \big( \dLocSys(X) \big)
	\]
\end{defi}

\begin{rema}
Notice that $C^*_{\dR} \big( \dLocSys(X) \big)$ admits, by construction, a natural mixed algebra structure. However, we will be mainly interested in the corresponding ''plain module'' and $\bE_\infty$-algebra structures
underlying the given mixed algebra structure on $C^*_{\dR} \big( \dLocSys(X) \big)$.
\end{rema}

\begin{prop}
Let $X$ be a proper and smooth scheme over an algebraically closed field of positive characteristic $p> 0$. We then have a well defined canonical morphism
	\[
		C^*_\dR \big( \rmB \anGLn \big) \otimes C^*_\emphet \big(X, \Ql \big)^\vee \to  C^*_\dR \big( \dLocSys(X) \big)
	\]
\end{prop}

\begin{proof}
Let $\rho \in \PerfSys(X)$ be a continuous representation. We have a canonical morphism
	\[
		\rmB \End(\rho) \to \rmB \End \big( \rho(*) \big)
	\]
in the \infcat $ \cS$, where $\rho(*)$ denotes the module underlying $\rho$. This association induces a well defined, up to contractible indeterminacy, morphism
	\[
		\PerfSys(X) \to \Perf^\an,
	\]
where $\Perf^\an \in \dSt \big( \dAfdl, \tau_\et \big)$ denotes the analytification of the algebraic stack of perfect complexes, $\Perf$. Therefore, we obtain a canonical morphism
	\begin{equation} \label{map:imp}
	 	f^* \colon \rmH \rmC \big( \Perf^\an \big) \otimes \rmH \big( \Perf^\an \big) \to \rmH \rmC \big( \PerfSys(X) \big) \otimes \rmH \big( \PerfSys(X) \big)
	\end{equation}
in the \infcat $\Mod_{\Ql}$, where $\rmH \big( \Perf^\an \big) \coloneqq \Map_{\Perf(\Ql)} \big( \Ql , \Ql \big) \simeq \Ql$ and $  \rmH \big( \PerfSys(X) \big) \simeq C^*_\et(X, \Ql)$. Thus we can rewrite \eqref{map:imp} simply as
	\begin{equation} \label{imp:map2}
		f^* \colon	\rmH \rmC \big( \Perf^\an \big) \to \rmH \rmC \big( \PerfSys(X) \big) \otimes C^*_\et(X, \Ql).
	\end{equation}
As \'etale cohomology $C^*_\et(X, \Ql) \in \Mod_{\Ql}$ is a perfect module we can dualize \eqref{imp:map2} to obtain a canonical morphism
	\[
		f^*  \colon \rmH \rmC \big( \Perf^\an \big) \otimes C^*_\et(X, \Ql ) \to  \rmH \rmC \big( \PerfSys(X) \big).
	\]
in the \infcat $\Mod_{\Ql}$. Consider now the commutative diagram
	\[
	\begin{tikzcd}
		\dLocSys(X) \ar{r} \ar{d}{j} & \rmB \anGLn \ar{d} \\
		\PerfSys(X) \ar{r} & \Perf^\an
	\end{tikzcd}
	\]
in the \infcat $\dSt \big( \dAfdl, \tau_\et \big)$. Then we have a commutative diagram at the level of loop stacks
	\[
	\begin{tikzcd}
		\Map \left( S^1, 	\dLocSys(X) \right) \ar{r}{i} \ar{d}{j} & \Map \left( S^1, \rmB \anGLn \right) \ar{d} \\
		\Map \left( S^1, \PerfSys(X) \right) \ar{r} & \Map \left( S^1, \Perf^\an \right).
	\end{tikzcd}
	\]
By taking global sections in the above diagram we conclude that the composite 
	\[
		f^* \circ i_!   \cH \cH \big( \cO_{ \rmB \anGLn} \big) \simeq f^* \circ i_! \cO_{\Map \left( S^1, 	\rmB \anGLn \right)}
	\]
has support in $\Map \left( S^1, 	\dLocSys(X) \right) \hookrightarrow \Map \left( S^1, 	\PerfSys(X) \right)$. Therefore, we can factor the composite
	\[
		\rmH \rmH \big( \rmB \anGLn \big)  \otimes C^*_\et(X, \Ql)^\vee \to \rmH \rmH \big( \Perf^\an \big)  \otimes C^*_\et(X, \Ql)^\vee \to \rmH \rmH \big( \PerfSys(X) \big)
	\]
as a morphism 
	\[
		\rmH \rmH \big( \rmB \anGLn \big)  \otimes C^*_\et(X, \Ql)^\vee \to  \rmH \rmH \big( \dLocSys(X) \big)
	\]
in the \infcat $\Mod_{\Ql}$. The analytic HKR theorem then provide us with the desired morphism
	\[
		C^*_\dR \big( \rmB \anGLn \big) \otimes C^*_\et \big(X, \Ql \big)^\vee \to  C^*_\dR \big( \dLocSys(X) \big)
	\]
in the \infcat $\Mod_{\Ql}$.
\end{proof}

\begin{rema}
A type GAGA theorem for reductive groups together with a theorem of B. Totaro, see \cite[Theorem 10.2]{totaro}, that the de Rham cohomology of the classifying stack $\anGLn$ coincides with $\ell$-adic cohomology
	\[
		C^*_\dR \big( \rmB \anGLn \big)  \simeq C^*_\dR \big( \rmB \GLn^{\mathrm{top}} \big)
	\]
in the \infcat $\Mod_{\Ql}$, where $\rmB \GLn^{\mathrm{top}}$ denotes the topological classifying stack associated to the general linear group $\GLn$. In particular, we obtain a morphism
	\[
		C^*_\et \big(\rmB \GLn, \Ql \big) \otimes C^*_{\et} \big(X, \Ql \big) \to C^*_\dR \big( \dLocSys(X) \big).
	\]
in the \infcat $\Mod_{\Ql}$. As $C^*_\dR \big( \dLocSys(X) \big)$ admits a natural $\bE_\infty$-algebra structure we obtain, by the universal property of the Sym construction, a well defined morphism
	\begin{equation} \label{mor:sp}
		\Sym \left( C^*_\et \big(\rmB \GLn, \Ql \big) \otimes C^*_{\et} \big(X, \Ql \big) \right) \to C^*_\dR \big( \dLocSys(X) \big).
	\end{equation}
in the \infcat $\CAlg_{\Ql}$. Assuming further that $X$ is a proper and smooth curve over an algebraically closed field, an $\ell$-adic version of Atiyah-Bott theorem proved in \cite{tamagawa}
implies that we can identify the left hand side of \eqref{mor:sp} with a morphism
	\[
		C^*_\et \left( \mathrm{Bun}_{\GLn} (X), \Ql \right) \to C^*_\dR \big( \dLocSys(X) \big)
	\]
in the \infcat $\CAlg_{\Ql}$.
\end{rema}

As a corollary we obtain:

\begin{coro}
Let $X$ be a smooth scheme over an algebraically closed field of positive characteristic $p > 0$. We have a canonical morphism
	\[
		\varphi \colon C^*_{\et} \big( \rmB \GLn, \Ql \big) \otimes C^*_\et \big( X, \Ql \big)^\vee \to C^*_{\dR} \big( \dLocSys(X) \big)
	\]
in the \infcat $\CAlg_{\Ql}$. Moreover, assuming further that $X$ is also a proper curve we obtain a canonical morphism
	\[
		C^*_\emphet \left( \mathrm{Bun}_{\GLn} \big(X \big), \Ql \right) \to  C^*_{\dR} \big( \dLocSys(X) \big)
	\]
in the \infcat $\CAlg_{\Ql}$.
\end{coro}

\begin{rema}
By forgetting the mixed $k$-algebra structure on $C^*_{\dR} \big( \dLocSys(X) \big)$ one can prove that the moprhism $\varphi$ sends the product of the canonical classes on $ C^*_{\et} \big( \rmB \GLn, \Ql \big) 
\otimes C^*_\et(X, \Ql)^\vee$ to the underlying cohomology class
of the shifted symplectic form $\omega $ on $\dLocSys(X)$.
\end{rema}

\end{document}